\documentclass[10pt]{article}

\usepackage{amsmath}
\usepackage{amsthm}
\usepackage{amssymb}
\usepackage{amsfonts}
\usepackage{graphicx}
\usepackage{subcaption}
\usepackage{verbatim}

\usepackage{tikz}
\usepackage{dsfont}
\usepackage{mathrsfs}

\theoremstyle{plain}
\newtheorem{theorem}{Theorem}[section]

\newtheorem{lemma}[theorem]{Lemma}

\theoremstyle{definition}

\newtheorem{remark}[theorem]{Remark}

\newcommand{\R}{\mathbb{R}}
\newcommand{\C}{\mathbb{C}}
\newcommand{\Z}{\mathbb{Z}}

\newcommand{\cI}{{\cal I}}

\newcommand{\eps}{\varepsilon}

\newcommand{\bx}{\bar{x}}

\newcommand{\cX}{\mathcal{X}}

\newcommand{\bydef}{\,\stackrel{\mbox{\tiny\textnormal{\raisebox{0ex}[0ex][0ex]{def}}}}{=}\,}

\newcommand{\tx}{\tilde{x}}

\newcommand{\Conj}{{\rm conj}}

\newcommand{\vect}[1]{{\boldsymbol{#1}}}

\newcommand{\bfFm}{{\vect{F}_{\!\vect{m}}}}
\newcommand{\bfIm}{{\vect{I}_{\!\vect{m}}}}
\newcommand{\bfk}{{\vect{k}}}
\newcommand{\bfs}{{\vect{s}}}
\newcommand{\bfq}{{\vect{q}}}
\newcommand{\bfp}{{\vect{p}}}
\newcommand{\bfm}{{\vect{m}}}
\newcommand{\bfn}{{\vect{n}}}

\newcommand{\bfl}{{\boldsymbol \ell }}
\newcommand{\bfj}{{\boldsymbol j }}

\title{Rigorous numerics for ill-posed PDEs: \\ periodic orbits in the Boussinesq equation}

\author{
Roberto Castelli\thanks{VU University Amsterdam, De Boelelaan 1081, 1081 HV Amsterdam,
The Netherlands (e-mail address: {\tt r.castelli@vu.nl}).}
\and Marcio Gameiro\thanks{Instituto de Ci\^{e}ncias Matem\'{a}ticas e de Computa\c{c}\~{a}o,
Universidade de S\~{a}o Paulo, Caixa Postal 668, 13560-970,
S\~{a}o Carlos, SP, Brazil (e-mail address: {\tt gameiro@icmc.usp.br}).
}
\and Jean-Philippe Lessard\thanks {
Universit\'e Laval, D\'epartement de Math\'ematiques et de Statistique, 1045 avenue de la
M\'edecine, Qu\'ebec, (Qu\'ebec), G1V 0A6, CANADA (e-mail address: {\tt jean-philippe.lessard@mat.ulaval.ca}).
}}

\date{}

\begin{document}

\maketitle 

\begin{abstract}
In this paper, we develop computer-assisted techniques for the analysis of periodic orbits of ill-posed
partial differential equations. As a case study, our proposed method is applied to the Boussinesq equation,
which has been investigated extensively because of its role in the theory of shallow water waves. The
idea is to use the symmetry of the solutions and a Newton-Kantorovich type argument (the radii polynomial
approach), to obtain rigorous proofs of existence of the periodic orbits in a weighted $\ell^1$ Banach
space of space-time Fourier coefficients with geometric decay. We present several computer-assisted
proofs of existence of periodic orbits at different parameter values.
\end{abstract}

\begin{center}
{\bf \small Keywords} \\ \vspace{.05cm}
{ \small Ill-posed PDE $\cdot$ periodic orbits $\cdot$ contraction mapping \\
Boussinesq equation $\cdot$ rigorous computations $\cdot$ interval analysis}
\end{center}

\begin{center}
{\bf \small Mathematics Subject Classification (2010)} \\ \vspace{.05cm}
{ \small 35R20 $\cdot$ 47J06 $\cdot$  65G40 $\cdot$ 65H20 $\cdot$ 35B10}
\end{center}

\section{Introduction}

Studying infinite dimensional nonlinear dynamical systems in the form of dissipative partial differential equations (PDEs), delay differential equations (DDEs) and infinite dimensional maps poses several difficulties. An issue is that it not possible in general to explore the dynamics in the entire infinite dimensional phase space. One common approach to circumvent this central difficulty is to focus on a set of {\em special} bounded solutions (e.g. fixed points, periodic orbits, connecting orbits between those) acting as organizing centers for the dynamics. Unfortunately, the nonlinearities in the models obstruct the analysis, and proving the existence of special solutions using standard {\em pen and paper} techniques may be an impossible task. In an effort to overcome these difficulties, the strengths of functional analysis, topology, algebraic topology (Conley index theory), numerical analysis, nonlinear analysis and scientific computing have recently been combined, giving rise to novel computer-assisted approaches to study infinite dimensional nonlinear problems. A growing literature on computational methods (functional analytic and topological)  is providing mathematically rigorous proofs of existence of special bounded solutions for PDEs \cite{MR1838755, MR2049869,MR2788972,MR2470145, MR1699017,MR2136516,MR2441958, MR2728184,MR2679365,MR2776917, MR3077902,MR2718657,jf_jb_OK,MR2019251, MR2220064,MR3353132,MR2821596, MR2378291,MR1849323,MR2652784,Castelli_Teismann}, DDEs \cite{MR2592879,MR2871794,minamoto} and infinite dimensional maps \cite{MR3124898,MR2067140, jay_CO_Kot_Shaffer}. Besides these recent successes in dynamical systems, ill-posed equations and problems with indefinite tails (e.g. strongly indefinite problems) have not received much attention in the field of rigorous computing. This is perhaps not surprising, as ill-posed equations do not naturally lead to the notion of a dynamical system. On the other hand, the idea of studying strongly indefinite problems by looking only for special type of bounded solutions (which exist for all time) has been commonplace for a while. Just think for instance of the theory of Floer homology which was originally developed to solve Arnold's conjecture\footnote{Arnold's conjecture states that the number of periodic solutions of a periodic Hamiltonian system is bounded from below by the topological invariants of the manifold on which the Hamiltonian system is defined.} (e.g. see \cite{MR0193645,MR987770,MR1045282}). One of the fundamental idea used by Floer in constructing its homology was precisely to restrict its attention to the set of bounded solutions. 

In the present paper we study ill-posed equations using a similar idea, that is we restrict our study to the space of periodic orbits. In this space, the solutions exist for all time, they are bounded and they are more regular than a solution of a typical initial value problem. This has the tremendous advantage of not having to care about the ill-posedness of the equation and about the ambient state-space. Using this point of view, we propose a computer-assisted, functional analytic approach to prove existence of periodic orbits in the ill-posed Boussinesq equation

\begin{align} \label{eq:boussinesq}
& u_{tt}= u_{yy} + \lambda u_{yyyy} +(u^2)_{yy}, \quad \lambda>0 \\
& u=u(t,y) \in \R, ~ y \in [0,1], ~t \in \mathbb{R},
\nonumber
\end{align}
which arises in the theory of shallow water waves. Equation \eqref{eq:boussinesq} is often called the {\em bad} Boussinesq equation, essentially because, as already mentioned in \cite{MR2538946}, it is not well-posed in any reasonable space, and one can find analytic initial conditions for which the solution is not defined (in almost any weak sense) on any interval of time (e.g. see \cite{MR795808,MR668408}). 

Let us mention that we are not the first to use the tools of rigorous computing to study ill-posed problems. In \cite{piotr_boussinesq}, the authors use the method of self-consistent bounds (e.g. see \cite{MR1838755,MR2049869,MR2788972}) to prove existence of periodic solutions of a Boussinesq-type equation perturbed by a time-dependent forcing term. However, the method of \cite{piotr_boussinesq} and our approach are quite different. First, we do not prove the existence of solutions which are obtained as perturbations of stationary points. Second, we prove existence of periodic orbits for the autonomous ill-posed system without a forcing term. Third, as we are only aiming at obtaining some particular bounded solutions (in this case periodic orbits), our functional analytic approach circumvents the ill-posedness of the Boussinesq equation: we do not need that the evolution is defined nor well-posed and we avoid doing a rigorous integration of the equation as in \cite{MR2049869,MR2788972,MR2679365}.

Before proceeding with the presentation of the method, let us sketch our general strategy for finding a periodic orbit of an ill-posed problem via a computer-assisted proof. Our approach is a natural extension of previous computer-assisted methods to study PDEs \cite{dlLFGL,time_periodic_PDEs, MR2338393,MR2718657,MR2776917}, and the following presentation closely follows the exposition in \cite{BL_notices}. We look for a periodic orbit, which we denote by $x$, and we introduce an equivalent formulation as one of the form $f(x)=0$. We look for solutions of this problem in a Banach space $(X,\| \cdot \|_X)$ which is a weighted $\ell^1$ space of space-time Fourier coefficients with geometric decay. We begin with a numerically obtained approximation $\bx$ having that $f(\bx) \approx 0$. Instead of solving $f(x)=0$ directly, we define a nonlinear operator $T$ whose fixed points are the zeros of $f$. The mapping $T$ is a Newton-like operator of the form $T(x) = x - A f(x)$, where the linear operator $A$ is chosen as an injective approximate inverse of $Df(\bx)$. We then show that $T$ is a contraction mapping on a closed ball $B_{r}(\bx) \subset X$ of radius $r>0$ and centered at $\bx$. To verify that $T$ is a contraction, we use a Newton-Kantorovich type argument (the radii polynomial approach), which provides an efficient way of obtaining a ball $B_{r}(\bx)$ on which $T$ is a contraction. The contraction mapping theorem yields the existence of a unique $\tx \in B_{r}(\bx)$ such that $T(\tx)=\tx$. The fixed-point $\tx$ corresponds to the wanted periodic orbit and since it belongs to $B_{r}(\bx)$, a rigorous error bound of the form $\|\tx-\bx\|_X \le r$ is obtained. 

The paper is organized as follows. In Section~\ref{sec:setup_symmetries}, we study the symmetries of the periodic  solutions of the Boussinesq equation \eqref{eq:boussinesq} and we derive an underdetermined system of the form $h(x)=0$, where $x$ corresponds to a periodic orbit. We then study some conserved quantities of \eqref{eq:boussinesq} which are used to fix the underdeterminacy of the system. Fixing one of the conserved quantities (the energy), we numerically compute periodic orbits. In Section~\ref{sec:rigorous_computation}, we modify the system $h=0$ into a simpler one of the form $f(x)=0$, where the energy is no longer fixed, but instead we fix the average of the solution. This choice simplifies the analysis of the computer-assisted proofs. Then, we introduce the radii polynomial approach to prove existence of periodic orbits close to numerical approximations, and all the necessary analytic bounds are derived. Finally, in Section~\ref{sec:results}, we present several computer-assisted proofs of existence of periodic orbits of \eqref{eq:boussinesq} at different parameter values $\lambda>0$.

\section{Set-up using the symmetry of the solutions} \label{sec:setup_symmetries}

For different values of $L$, we look for solutions of \eqref{eq:boussinesq} that are $\frac{2\pi}{L}$-periodic in time and $1$-periodic in space, that is
$u(t+\frac{2 \pi}{L}, y) = u(t, y)$ and  $u(t, y) = u(t, y+1)$ for all $y \in [0,1]$ and $t \in \mathbb{R}$. Moreover we assume that the solution satisfies an even/odd symmetry both in time and space. As mentioned in \cite{MR2538946}, the space of spatially symmetric solutions is invariant. Therefore we plan on studying the Boussinesq equation \eqref{eq:boussinesq} supplemented with the even periodic boundary conditions $u(t, -y) = u(t, y)$. Moreover, we will impose yet another symmetry in time to simplify the search. As the next remark demonstrates, only a specific type of space-time symmetry may exist. 

\begin{remark} \label{rem:even/even} 
Among all possible combination even-odd/even-odd symmetries, non trivial solutions exist only in the case even/even. 
For example, by an odd/even symmetry, we mean a solution $u(t,y)$ satisfying $u(-t,y)=-u(t,y)$ and $u(t,-y)=u(t,y)$ for all $y \in [0,1]$ and $t \in \mathbb{R}$.
If $u(t,y)$ is symmetric both in time and space, whatever the symmetry is, the product $u^2(t,y)$ is even both in time and space.  The second derivative $(u^2)_{yy}$ preserves the even symmetry in space and clearly the same symmetry in time. Hence   $(u^2)_{yy}$ is necessarily even-even. All the $\partial _{tt}$, $\partial _{yy}$
and $\partial _{yyyy}$ preserve the symmetry of the function they apply to, therefore the only possibility for the solution $u(t,y)$ to be symmetric is that it is even/even.
 \end{remark}
 
The space-time periodic solutions of \eqref{eq:boussinesq} can be expanded using the Fourier expansion 
\begin{equation} \label{eq:space-time-expansion}
u(t,y) = \sum_{\bfk \in \mathbb{Z}^2} c_\bfk \psi_\bfk(t,y),\quad  \text{where  }  \psi_\bfk(t,y) \bydef e^{i L k_1t}e^{i 2 \pi k_2 y}.
\end{equation}
Following Remark~\ref{rem:even/even}, we look for periodic solutions $u(t,y)$ of \eqref{eq:boussinesq} satisfying the even/even symmetries
\[
u(t,-y)=u(t,y)
\qquad \text{and} \qquad
u(-t,y)=u(t,y).
\]

Since we are interested in real solutions and because of the symmetries, the Fourier coefficients $c_\bfk$ in \eqref{eq:space-time-expansion} satisfy the relations
\begin{equation} \label{eq:relations_c_k_1}
\begin{array}{lll}
c_{-k_1, -k_2} = \Conj(c_\bfk) \\
c_{k_1, -k_2} = c_\bfk \\
c_{-k_1, k_2} = c_\bfk,
\end{array}
\end{equation}
where given a complex number $z = a + i b \in \C$, $\Conj(z) = a - i b$ denotes the complex conjugate.
From \eqref{eq:relations_c_k_1} we get that $\Conj(c_\bfk)=c_\bfk$, which implies that $c_\bfk \in \R$, and also we get that
\begin{equation} \label{eq:symmetry}
c_{\pm k_1, \pm k_2} = c_\bfk.
\end{equation}

Clearly, the even/even  solution $u(t,y)$ can be expanded on the basis of cosine with the index $ \bfk$ ranging in $\mathbb N^2$. However, it is convenient to keep the expansion \eqref{eq:space-time-expansion} and the constraints \eqref{eq:symmetry}. In this setting the  product $u^2(t,y)$ can be easily expanded as
\begin{equation} \label{eq:convolution}
u^2(t,y)=\sum_{\bfk\in \Z^2}(c^2)_\bfk\psi_\bfk(t,y),\quad  \text{where  }
(c^2)_\bfk \bydef \sum_{\bfl+\bfj = \bfk} c_{\bfl} c_{\bfj}.
\end{equation}
The Boussinesq equation \eqref{eq:boussinesq} can be re-written as $u_{tt} - \left( u +\lambda u_{yy} +u^2 \right)_{yy} = 0$. Note that 
\[
u +\lambda u_{yy} +u^2 =  \sum_{\bfk \in \mathbb{Z}^2} \left(  [1-\lambda k_2^2(2 \pi)^2] c_\bfk + (c^2)_\bfk \right) \psi_\bfk.
\]
Hence, the Fourier coefficients of the expansion of $u$ given by \eqref{eq:space-time-expansion} plugged in the Boussinesq equation are
\begin{equation}\label{eq:hk}
h_\bfk \bydef k_1^2 L^2 c_\bfk - k_2^2 (2 \pi)^2 \left(  [1-\lambda k_2^2(2 \pi)^2] c_\bfk + (c^2)_\bfk \right) = \eta_\bfk c_\bfk - 4 \pi^2 k_2^2 (c^2)_\bfk,
\end{equation}
where
\[
\eta_\bfk =\mu_\bfk(L,\lambda) \bydef k_1^2 L^2 +  16 \pi^4 \lambda k_2^4 - 4 \pi^2 k_2^2.
\]
Looking for periodic solutions of \eqref{eq:boussinesq} that are $\frac{2\pi}{L}$-periodic in time and $1$-periodic in space 
is equivalent to solve $h_\bfk=0$ for any $\bfk \in \Z^2$. From conditions \eqref{eq:symmetry} it is straightforward to verify that   
\begin{equation} \label{eq:relations_h_k}
\begin{array}{lll}
h_{\pm k_1, \pm k_2} = h_\bfk.
\end{array}
\end{equation}
Hence, finding even/even $\frac{2\pi}{L}$-periodic in time,  and $1$-periodic in space periodic solutions of the Boussinesq equation is equivalent to looking for solutions of  $h_\bfk=0$ for all $\bfk=(k_1,k_2)$ with $k_1,k_2 \geq 0$ in the unknowns $\{c_\bfk\}_{\bfk\geq 0}$, subjected to the conditions \eqref{eq:symmetry}.

Looking at \eqref{eq:hk} we immediately realize that
$$
h_{0,0}\equiv 0,\quad h_{k_1,0}=k^2_1L^2c_{k_1,0},
$$
and so $c_{k_1,0}=0$ for any $k_1>0$. The first relation shows that the system $h=0$ is underdetermined. Therefore we will need either to add one more equation or to remove one of the unknowns. 
In order to numerically find some initial periodic orbits, we will use the conserved quantities of the Boussinesq equation to fix the fact that the system is underdetermined.

\subsection{Conserved quantities and integrals of motion}

Following \cite{MR957220} with the necessary adaptations, we see that the system has integral of motions. If $u(t,y)$ is a time $\frac{2\pi}{L}$-periodic and space $1$-periodic solution of \eqref{eq:boussinesq} then the following quantities
\[
J(t) \bydef \int_0^{1}u_t(t,y) dy
\quad \text{and} \quad
W(y) \bydef \int_0^{\frac{2\pi}{L}} \left( u(t,y) +\lambda u_{yy}(t,y) +[u(t,y)]^2 \right)_y dt
\]
are conserved, as we demonstrate next.

Let us first study the conserved quantity $J(t)$. Since $u$ is $1$-periodic in space, then
\[
\frac{d}{dt}J(t)  =\int_0^1u_{tt}dy=\int_0^{1} \left( u +\lambda u_{yy} +u^2 \right)_{yy}dy=0,
\]
which shows that $J(t)$ is conserved along the solutions of \eqref{eq:boussinesq}. Expanding
$$
J(t)=\int_0^{1}\sum_{\bfk \in \Z^2}c_\bfk k_1e^{ik_1Lt}e^{ik_22\pi y} dy=\sum_{\bfk\in \Z^2}c_\bfk k_1e^{ik_1Lt}\int_0^{1}e^{ik_22\pi y}dy=\sum_{k_1\in\Z}c_{k_1,0}k_1e^{ik_1Lt},
$$
and using that $J(t)$ is invariant in time, it follows that $c_{k_1,0}=0$ for any $k_1\neq 0$.
Similarly, we obtain the same conclusion about the $h_\bfk$, namely that  $h_{k_1,0}=0$ for any $k_1\neq 0$, as $h_{k_1,0}=0\Leftrightarrow k_1^2L^2c_{k_1,0}=0$.
This means that the integral of motion $J$ is encoded in the equation $h_{k_1,0}=0$ and results in $c_{k_1,0}=0$ for any $k_1>0$.
Moreover, note that $J=J(t)=0$ for all $t \in \R$. Indeed, assume that $J \ne 0$ and denote  $I(t)=\int_0^1 u dy$. Then $I'(t)=\int_0^1 u_t dy = J(t) = J \in \R$ and integrating leads to 
$I(t)=Jt+constant$. This implies that $\int_0^1 u dy \to \infty$ as $t$ grows. In other words, this means that the average of the periodic solution $u$ of \eqref{eq:boussinesq} is unbounded as $t$ grows, which contradicts the fact that the solution $u(t,y)$ to be bounded.

Let us now study the conserved quantity $W(y)$. Since $u$ is $\frac{2\pi}{L}$-periodic in time, then
\[
\frac{d}{dy}W(y) =\int_0^{\frac{2\pi}{L}}(u +\lambda u_{yy} +u^2)_{yy} dt=\int_0^{\frac{2\pi}{L}}u_{tt} dt=0,
\]
which shows that $W(y)$ is conserved along the solutions of \eqref{eq:boussinesq}. Proceeding as before, that is writing $W(y)$ in terms of the Fourier coefficients and assuming that $W$ does not depend on $y$, we end up with the relation $h_{0,k_2}=0$ for any $\bfk=(0,k_2)$ with $k_2\geq 0 $.

We now discuss the Hamiltonian structure of the system. Let us introduce the momentum
$$
v(t,y)\bydef u_t(t,0)+\int_0^yu_t(t,\xi)\, d\xi.
$$
The differential equation in \eqref{eq:boussinesq} is equivalent to the system
\[
\begin{cases}
u_t=v_y\\
v_t=\lambda u_{yyy}+u_y+(u^2)_y.
\end{cases}
\]
Indeed $u_{tt}=(v_y)_t=(v_t)_y=\lambda u_{yyyy}+u_{yy}+(u^2)_{yy}$.
Denote the {\em energy functional} by
$$
E(v,u)=T(v)+V(u)\bydef \int_0^1\frac{1}{2}v^2 dy+\int_0^1 \left( - \frac{\lambda}{2}(u_y)^2+\frac{1}{2}u^2+\frac{1}{3}u^3 \right) dy
$$
\begin{lemma}
If $u(t,y)$ is a space-time periodic solution of \eqref{eq:boussinesq} with the period as above, and $v(t,y)$ defined as above, then 
$$
\frac{d}{dt}E(v,u)=0.
$$
\end{lemma}

\begin{proof}
Integration by parts leads to
\begin{align*}
\frac{d}{dt}E(v,u)&=\int_0^1vv_t\,  dy+\int_0^1 -\lambda u_yu_{yt}+uu_t+u^2u_t \, dy\\
&=\int_0^1v(\lambda u_{yyy}+u_y+(u^2)_y)\,  dy+\int_0^1 \lambda u_{yy}u_{t}+uu_t+u^2u_t \, dy\\
&=-\int_0^1v_y(\lambda u_{yy}+u+u^2)\,  dy+\int_0^1 u_t(\lambda u_{yy}+u+u^2) \, dy=0. \qedhere
\end{align*}
\end{proof}

We now rephrase the conservation of  the energy in terms of the Fourier coefficients of the solution $u(t,y)$. In the following we again assume that $u(t,y)$ is a space-time periodic solution and we also assume that $u(t,y)$ is even both in time and space. That is, we assume that the third relation of \eqref{eq:relations_c_k_1} holds.

Since the energy is conserved, we have that $E=E(t)=E(0)$. First compute
\begin{align*}
v(0,y)&=u_t(0,0)+\int_0^yu_t(0,\xi)\ d\xi\\
& =\sum_{\bfk\in\Z^2}c_\bfk ik_1L+\int_0^y \sum_{\bfk\in \Z^2}c_\bfk ik_1Le^{i2\pi k_2\xi}\, d\xi \\
& =\sum_{\bfk\in \Z^2}c_\bfk ik_1L\left(1+\int_0^y e^{i2\pi k_2\xi}\, d\xi\right).
\end{align*}
Since $c_{-k_1,k_2}=c_{k_1,k_2}$, the sum vanishes for any $y$. Therefore $v(0,y)=0$ for any $y$. 
This implies that term $T(v)$ in the energy vanishes, and so
$$
E=\int_0^1 \left( -\frac{\lambda}{2}(u_y(0,y))^2+\frac{1}{2}u(0,y)^2+\frac{1}{3}u(0,y)^3 \right) dy.
$$
Using \eqref{eq:convolution} and
\begin{align*}
u_y^2(0,y)&=-\sum_{\bfk\in \Z^2}\sum_{\bfl+\bfj=\bfk} (2 \pi \ell_2 c_{\bfl}) (2 \pi j_2 c_{\bfj}) e^{ik_2 2\pi y}
\\
u^2(0,y)&=\sum_{\bfk\in \Z^2}\sum_{\bfl+\bfj=\bfk}c_{\bfl}c_{\bfj} e^{ik_2 2\pi y}
\\
u^3(0,y)&=\sum_{\bfk\in \Z^2} \left( \sum_{\bfl+\bfj+\bfn=\bfk}c_{\bfl}c_{\bfj}c_{\bfn} \right) e^{ik_2 2\pi y},
\end{align*}
%
it follows that
$$
\begin{aligned}
E&=\sum_{\bfk\in \Z^2\atop k_2=0}\left[\frac{\lambda}{2} \sum_{\bfl+\bfj=\bfk} (4 \pi^2 \ell_2 j_2) c_{\bfl} c_{\bfj} + \frac{1}{2}\sum_{\bfl+\bfj=\bfk}c_{\bfl}c_{\bfj} 
+ \frac{1}{3} \sum_{\bfl+\bfj+\bfn=\bfk} c_{\bfl} c_{\bfj} c_{\bfn} \right] \\
&=\sum_{\bfk\in \Z^2\atop k_2=0} 2 \lambda \pi^2 (\alpha*\alpha)_\bfk+\frac{1}{2}(c*c)_\bfk+\frac{1}{3}(c*c*c)_\bfk
\end{aligned}
$$
where $\alpha_\bfk \bydef c_\bfk k_2$. Fixing a value for $E$, we replace the equation $h_{0,0}$ by $$\sum_{\bfk\in \Z^2\atop k_2=0} 2 \lambda \pi^2(\alpha*\alpha)_\bfk+\frac{1}{2}(c*c)_\bfk+\frac{1}{3}(c*c*c)_\bfk-E$$  to remove the underdeterminacy of the system.

\begin{remark}
We use the augmented system, that is the system with the energy relation instead of $h_{0,0}$ for the numerical computation of the solution. Once the numerical solution is obtained,  for the validation we fix $c_{0,0}$ and solve for the other coefficients of the system $h_\bfk$ for $\bfk\neq 0$.
\end{remark}

\section{The rigorous computational method} \label{sec:rigorous_computation}
  
The first step of the rigorous computational method consist of transforming the problem of looking for periodic orbits of \eqref{eq:boussinesq} into an equivalent problem of the form $f(x)=0$.
In this process, we need to make sure that the solutions of $f=0$ will be locally isolated, as we aim at using the contraction mapping theorem to prove their existence.

\subsection{Defining the operator equation \boldmath $f(x)=0$ \unboldmath}

By setting $c_{k_1,0}=0$ for all $k_1\neq 0$, an even/even solution of the PDE is given by solving 
$$
h_\bfk=0,\quad \text{for all } k_1\geq 0 , k_2>0
$$
for the coefficients $\{c_{0,0}, \{c_\bfk\}_{k_1\in \Z,k_2\in \Z\setminus\{0\}}\}$ with the restrictions $c_{\pm k_1,\pm k_2}=c_{k_1,k_2}$. 
Later on we remove $c_{0,0}$ from the unknowns and, because of the symmetry constraints, we solve the system for the coefficients  $c_\bfk=c_{k_1,k_2}$ with $k_1\geq 0$ and $k_2>0$ only. 
This motivates the following definitions. Denote 
\[
 \mathcal Z^2_+ \bydef \{ \bfk = (k_1,k_2): k_1\geq 0, k_2>0\},
\]
and introduce the spaces
$$
X \bydef \Big\{x=\{x_\bfk\}_{\bfk\in \mathcal Z^2_+}\Big\}
\quad  \text{and} \quad
\mathcal X \bydef \Big\{ x=\{x_\bfk\}_{\bfk\in \Z^2} \Big\}.
$$
Defined the function
$$
\begin{aligned}
{\rm sym}: ~ & X\to \mathcal X\\
&x\mapsto {\rm sym}(x)=x_{sym}
\end{aligned}
$$
by
$$
(x_{sym})_{k_1,0}=0, \quad (x_{sym})_{\pm k_1,\pm k_2}=x_\bfk, \quad \text{for all } \bfk\in \mathcal Z^2_+.
$$
Note that $(x_{sym})_{0,0}=0$. Define
$$
X_{sym}=\Big\{ {\rm sym}(x) : x\in X\Big\}\subset \mathcal X.
$$
Fix $c_{0,0}\in \R$. 
\begin{itemize}
\item Given $x\in \mathcal X$, let $y \bydef c_{0,0}+x \in \mathcal X$ given by $y_{0,0}=c_{0,0}+x_{0,0}$ and $y_\bfk=x_\bfk$ for all $\bfk\neq (0,0).$
\item For $x\in X$, denote 
\begin{equation} \label{eq:c(x)}
c(x) \bydef c_{0,0}+{\rm sym}(x) \in \mathcal X.
\end{equation}
\item For any $x,y\in \mathcal X$ denote by $x*y$ the standard convolution product $$(x*y)_\bfk \bydef \sum_{\bfl+\bfj=\bfk}x_\bfl y_\bfj.$$
\item For any $x\in X$ denote 
\begin{equation} \label{eq:c^2(x)}
c^2(x) \bydef c(x)*c(x).
\end{equation}
\item Given $x\in \mathcal X$, define $(c_{0,0}*x)$ component-wise as $(c_{0,0}*x)_\bfk=c_{0,0}x_\bfk$.
\end{itemize}
The last definition and indeed all the definitions involving $c_{0,0}$ follow immediately from the standard operations between sequences, once the scalar $c_{0,0}$ is seen as a sequence where all but the $(0,0)$ coefficients are zero.

For a multi-index with positive entries $\bfm=(m_1,m_2)$, denote 
\[
\bfFm \bydef \{\bfk\in \mathcal Z^2_+: k_1<m, k_2<m\} \qquad \text{and} \qquad \bfIm \bydef \mathcal Z^2_+\setminus \bfFm.
\]
For a set $X$ or a sequence $x$ we often adopt the notation $X_\bfFm$ and $x_\bfFm$ to denote the restriction to those elements whose index $\bfk$ belongs to $\bfFm$. The same is done for $\bfIm$. For instance $X_\bfFm=\{x=\{x_\bfk\}_{\bfk\in \bfFm}\}$.

%

Assume that a numerical solution $\{\bar c_{0,0},\bar x=\{\bar x_\bfk\}_{\bfk\in \bfFm}\}$ has been computed so that $h_\bfk\approx 0 $ for any $\bfk\in \bfFm$, as discussed in the previous section.
Now, let us set $c_{0,0}=\bar c_{0,0}$ and remove $c_{0,0}$ from the set of unknowns. 
 Then according to the previous definitions, the problem consists of finding $x\in X$ so that 
$$
h(x)=0,\quad x=\{x_\bfk\}_{\bfk\in \mathcal Z^2_+},\quad h=\{h_\bfk\}_{\bfk\in \mathcal Z^2_+}
$$
where
$$
h_\bfk(x)=\mu_\bfk x_\bfk - 4 \pi^2 k_2^2 (c^2(x))_\bfk.
$$

Recall \eqref{eq:hk}. Since we only need to solve $h_\bfk=0$ for $k_2> 0$, we can divide it by $4\pi^2k^2_2$.  Replace the previous $h_\bfk$ with
\begin{equation} \label{eq:f_k}
f_\bfk(x) \bydef \left(\frac{L^2}{4\pi^2}\frac{k_1^2}{k_2^2}+4\pi^2\lambda k_2^2-1\right)x_{\bfk}-(c^2(x))_\bfk=\mu_\bfk x_\bfk-(c^2(x))_\bfk
\end{equation}
where
\begin{equation} \label{eq:new_mu_k}
\mu_\bfk \bydef \mu_\bfk(L,\lambda)=\frac{L^2}{4\pi^2}\frac{k_1^2}{k_2^2}+4\pi^2\lambda k_2^2-1.
\end{equation}
From \eqref{eq:f_k}, we define the problem 
\begin{equation} \label{eq:f=0}
f(x) =0,
\end{equation}
where $f=\{f_{\bfk} \}_{\bfk \in \mathcal Z^2_+}$. The rest of the paper consist of developing a computer-assisted approach to find solutions 
of \eqref{eq:f=0}. The first step is to define a Banach space in which we look for solutions. 

\subsection{Norms and Banach space}

For a choice of $\nu>1$, we endow $X$ and $\mathcal X$ with the norm
$$
\|x\|_\nu\bydef \sum_{\bfk\in \mathcal Z^2_+}|x_\bfk|\nu^{|\bfk|}
\qquad \text{and} \qquad
\|x\|_\nu^*\bydef  \sum_{\bfk\in \Z^2}|x_\bfk|\nu^{|\bfk|},
$$
respectively, where, given $\bfk=(k_1,k_2)$, we use the standard notation $|\bfk|=|k_1|+|k_2|$. The second one is the usual geometric norm defined on the space of bi-infinite sequences. 
Accordingly, denote the closed ball in $X$ with radius $r$ centered at the origin by
$$
B(r)=\{ x\in X: \|x\|_\nu\leq r\}.
$$ 

\begin{lemma} \label{lem:norm_comparisons}
For any $x\in X$
$$
\|{\rm sym}(x)\|_\nu^*\leq 4 \| x\|_\nu.
$$
\end{lemma}

\begin{proof}
\begin{align*}
\|{\rm sym}(x)\|_\nu^*&=\sum_{|k_1|>0,|k_2|>0}|(x_{sym})_\bfk|\nu^{|\bfk|}+\sum_{k_1\in \Z,k_2=0}|(x_{sym})_\bfk|\nu^{|\bfk|}+\sum_{k_1=0,|k_2|>0}|(x_{sym})_\bfk|\nu^{|\bfk|}\\
&=4\sum_{k_1>0,k_2>0}|x_\bfk|\nu^{|\bfk|}+2\sum_{k_1=0,k_2>0}|x_\bfk|\nu^{|\bfk|} \\
& \le 4\sum_{\bfk\in \mathcal Z^2_+}|x_\bfk|\nu^{|\bfk|}  =4\|x\|_\nu. \qedhere
\end{align*}
\end{proof}

\begin{remark}
It is often useful to consider the norm $\|\cdot\|_\nu$ as a norm on $X_{sym} \subset \cX$. To make sense of this, note that any $x = \{x_\bfk\}_{\bfk \in \Z^2} \in X_{sym}  \subset \cX$ satisfies the symmetries $x_{k_1,0}=0$ and $x_{\pm k_1,\pm k_2}=x_\bfk$, for all $\bfk \in \Z^2$, and therefore is entirely defined over the indices $\bfk \in \mathcal Z^2_+$. Hence, we can consider the norm on $X_{sym}$ as 
\[
\| x \|_\nu = \sum_{\bfk\in \mathcal Z^2_+}|x_\bfk|\nu^{|\bfk|}
\]
\end{remark}

\begin{lemma}
For any $x\in X$
$$
\|c_{0,0}*{\rm sym}(x)\|_\nu=|c_{0,0}|\|x\|_\nu.
$$
\end{lemma}

\begin{proof}
This follows using that $(c_{0,0}*{\rm sym}(x))_\bfk=c_{0,0}(x_{sym})_\bfk$ and $\|x_{sym}\|_\nu=\|x\|_\nu$. \qedhere
\end{proof}

\begin{lemma}
For any $x,y\in X$
$$
\|{\rm sym}(x)*{\rm sym}(y)\|_\nu\leq 16\|x\|_\nu\|y\|_\nu. 
$$
\end{lemma}

\begin{proof}
Denoting $x_{sym} = {\rm sym}(x)$ and $y_{sym} = {\rm sym}(y)$, we get that
\begin{align*}
\|x_{sym}*y_{sym}\|_\nu&=\sum_{\bfk\in \mathcal Z^2_+}\left| \sum_{\bfj+\bfl=\bfk\atop \bfj,\bfl\in \Z^2}(x_{sym})_\bfj(y_{sym})_\bfl\right|\nu^{|\bfk|} \\
& =\sum_{\bfk\in \mathcal Z^2_+}\left| \sum_{\bfj\in \Z^2}(x_{sym})_\bfj(y_{sym})_{\bfk-\bfj}\right|\frac{\nu^{|\bfk|}}{\nu^{|\bfj|}}\nu^{|\bfj|} \\
&\le \sum_{\bfj\in \Z^2}|(x_{sym})_\bfj|\nu^{|\bfj|}\sum_{\bfk\in \mathcal Z^2_+}|(y_{sym})_{\bfk-\bfj}|\nu^{|\bfk|-|\bfj|} \\
& \le  \sum_{\bfj\in \Z^2}|(x_{sym})_\bfj|\nu^{|\bfj|} \sum_{\bfk\in \mathcal Z^2_+}|(y_{sym})_{\bfk-\bfj}|\nu^{|\bfk-\bfj|} \\
& \le  \sum_{\bfj\in \Z^2}|(x_{sym})_\bfj|\nu^{|\bfj|} \sum_{\bfk\in \Z}|(y_{sym})_{\bfk-\bfj}|\nu^{|\bfk-\bfj|} \\
&= \|x_{sym}\|_\nu^*\|y_{sym}\|_\nu^*  \leq 16\|x\|_\nu\|y\|_\nu.  \qedhere
\end{align*}
\end{proof}

We are now interested in studying linear functionals and linear operators acting on $X$. 
Given a linear operator $\mathcal L:X\to X$, define
$$
||\mathcal L||=\sup_{\|x\|_\nu=1}\|\mathcal L(x)\|_\nu.
$$
It readily follows that $\|\mathcal L(x)\|_\nu\leq ||\mathcal L||\, \|x\|_\nu$ for any $x\in X$. 

A linear operator $\mathcal L$ on $X$ is determined by the action of $\mathcal L$ on the components of $x\in X$.
Thus, associated to an operator $\mathcal L$ there is a uniquely defined matrix of operators, still denoted by $\mathcal L=\{\mathcal L_{(\bfk,\bfj)}\}_{(\bfk,\bfj)\in \mathcal Z^2_+\times \mathcal Z^2_+}$, so that
$$
(\mathcal L(x))_\bfk=\sum_{\bfj\in \mathcal Z^2_+}\mathcal L_{(\bfk,\bfj)}x_{\bfj}.
$$ 
Treating each row of $\mathcal L$ as a linear functional and using that the dual space of $X$ is a weighted $\ell^\infty$ space, it follows that the operator norm is given by
\begin{equation} \label{eq:oper_norm}
||\mathcal L||=\sup_{\bfj\in \mathcal Z^2_+}\frac{1}{\nu^{|\bfj|}}\sum_{\bfk\in \mathcal Z^2_+}|\mathcal L_{(\bfk,\bfj)}|\nu^{|\bfk|}.
\end{equation}
We refer to \cite{HLM} for more details on how to compute such operator norms using the theory of dual spaces.
Next, we introduce the Newton-like operator whose fixed points correspond to the wanted periodic orbits. 

\subsection{The Newton-like operator \boldmath $T(x)=x-Af(x)$ \unboldmath}

The construction of $T$ begins by assuming the existence of $\bar x=\{\bar x_\bfk\}_{\bfk\in \bfFm}$ such that $f_\bfk(\bx) \approx 0 $ for any $\bfk\in \bfFm$. 
The Newton-like operator is defined as $T(x)=x-Af(x)$, where $A$ is a carefully chosen approximate inverse for $Df(\bx)$. Next we introduce the definition of $A$.

Recalling \eqref{eq:c(x)} and the definition of $c^2(x)$ in \eqref{eq:c^2(x)}, let us now compute the derivative of $f_\bfk$.
Given $x,v\in X$,
\begin{equation}\label{eq:derh}
\begin{aligned}
Df_\bfk(x)(v)&=\lim_{\epsilon\to 0}\frac{h_\bfk(x+\epsilon v)-h_\bfk(x)}{\epsilon} \\
&=\mu_\bfk v_\bfk - \lim_{\epsilon\to 0} \frac{1}{\epsilon}\Big[ c^2(x+\epsilon v)-c^2(x)\Big]_\bfk\\
&=\mu_\bfk v_\bfk-\lim_{\epsilon\to 0}\frac{1}{\epsilon}\Big[ c(x+\epsilon v)*c(x+\epsilon v)-c(x)*c(x)\Big]_\bfk\\
&=\mu_\bfk v_\bfk-\lim_{\epsilon\to 0}\frac{1}{\epsilon}\Big[ (c_{0,0}+{\rm sym}(x)+\epsilon {\rm sym}(v))*(c_{0,0}+{\rm sym}(x)+\epsilon {\rm sym}(v))\\
&\qquad \qquad \qquad \qquad -(c_{0,0}+{\rm sym}(x))*(c_{0,0}+{\rm sym}(x))\Big]_\bfk\\
&=\mu_\bfk v_\bfk-2\Big[(c_{0,0}+{\rm sym}(x))*{\rm sym}(v)\Big]_\bfk.
\end{aligned}
\end{equation}
By writing explicitly the components of  ${\rm sym}(v)$ in terms of those of $v$, it follows that the entries of the Jacobian matrix of $h$ with respect to $x$ have the form 
$$
\frac{\partial f_\bfk}{\partial x_\bfj}(x)=\mu_\bfk\delta_{\bfk,\bfj}-2\mathcal C_{\bfk,\bfj}(x),\qquad \bfk, \bfj\in \mathcal Z^2_+
$$
where

\begin{equation}\label{eq:Ckj}
\mathcal C_{\bfk,\bfj}(x) 
=\left\{\begin{array}{ll}
(x_{\bfk-\bfj}+x_{\bfk+\bfj}), &j_1=0\\
(x_{\bfk-\bfj}+x_{\bfk+\bfj}+x_{\bfk-(j_1,-j_2)}+x_{\bfk+(j_1,-j_2)}), &j_1>0.
\end{array}\right.
\end{equation}
Denote by $Df(x)$ the Jacobian of $f$ at $x$, that is $Df(x)_{(\bfk,\bfj)}= \frac{\partial f_\bfk}{\partial x_\bfj}(x)$, and $D^{(\bfm)} f (\bx)$ the Jacobian of $f_{\bfFm}$ with respect to $x_{\bfFm}$ at $\bar x$, that is $D^{(\bfm)} f (\bx)=\{D^{(\bfm)} f(\bx)_{(\bfk,\bfj)}\}_{\bfk,\bfj\in \bfFm}$ where $D^{(\bfm)} f(\bx)_{(\bfk,\bfj)}=\frac{\partial f_\bfk}{\partial x_\bfj}(\bar x)$.

Let also $A^{(\bfm)}=\{ A^{(\bfm)}_{(\bfk,\bfj)}\}_{\bfk,\bfj\in F_m}$ be an approximate inverse of $D^{(\bfm)} f$ and define the linear operator $A=\{A_{(\bfk,\bfj)}\}_{(\bfk,\bfj)\in\mathcal Z^2_+}$ component-wise by
\begin{equation} \label{eq:A}
A_{(\bfk,\bfj)}= 
\begin{cases}
A^{(\bfm)}_{(\bfk,\bfj)},\quad &\bfk,\bfj\in \bfFm\\
\mu_\bfk^{-1},&\bfk=\bfj, \bfk\in \bfIm\\
0,&{\rm otherwise} .
\end{cases}
\end{equation}

Define the Newton-like operator $T:X\to X$ by
\begin{equation} \label{eq:T}
T(x)=x-Af(x). 
\end{equation}

\subsection{The radii polynomial approach}

Consider $\bx \in X$ and denote 
\[
B_r(\bx) \bydef \bx + B(r) = \{ x \in X: \|x -\bx \|_\nu\leq r\}
\]
the closed ball in $X$ of radius $r>0$ with center $\bx$. In general, $\bx$ is a point that is an approximation solution of $f(x)=0$, typically obtained via Newton's method.

Let $Y$, $Z_0$, $Z_1$ and $Z_2$ be bounds satisfying 
\begin{align*}
\|T(\bar x)-\bar x\|_\nu=\|Af(\bar x)\|_\nu & \le Y \\
\sup_{u\in B(1)}\|(I-AA^\dag)u\|_\nu & \le Z_0, \\
\sup_{u,v\in B(1)}\|A(Df(\bar x+rv)-A^\dag)u\|_\nu & \le Z_1+Z_2r,
\end{align*}
where the linear operator $A^\dag=\{A^\dag_{(\bfk,\bfj)}\}_{(\bfk,\bfj)\in\mathcal Z^2_+}$ is defined component-wise as
\[
A^\dag_{(\bfk,\bfj)}
=
\begin{cases}
D^{(\bfm)}f_{(\bfk,\bfj)}(\bx),\quad &\bfk,\bfj\in \bfFm\\
\mu_\bfk,&\bfk=\bfj, \bfk\in \bfIm\\
0,&{\rm otherwise} .
\end{cases}
\]

Once the bounds $Y$, $Z_0$, $Z_1$ and $Z_2$ have been found, define the {\em radii polynomial} by
\begin{equation} \label{eq:radii_polynomial}
p(r)=Y+(Z_0+Z_1-1)r+Z_2 r^2.
\end{equation}

\begin{lemma} \label{lem:rad_poly}
Assume that the linear operator $A$ defined component-wise by \eqref{eq:A} is injective. Let $p(r)$ the radii polynomial given by \eqref{eq:radii_polynomial}. If there exists $r>0$ such that $p(r)<0$, then 
there exists a unique $\tx \in B_r(\bx)$ such that $f(\tx)=0$.
\end{lemma}

\begin{proof}
See the proof of Proposition~1 in \cite{HLM}.
\end{proof}

The radii polynomial approach therefore consists of constructing explicitly the polynomial $p(r)$ defined in \eqref{eq:radii_polynomial}, to find $r>0$ such that $p(r)<0$, and to apply Lemma~\ref{lem:rad_poly} to obtain a true solution $\tx \in B_r(\bx)$ such that $f(\tx)=0$.

To perform the computation of the bounds $Y$, $Z_0$, $Z_1$ and $Z_2$ involved in the radii polynomial, we assume that the finite dimensional parameter $\bfm=(m_1,m_2)$ satisfies the condition
\begin{equation}\label{eq:cond_m}
m_2 \geq \max \left\{ m_1, \frac{L}{2\pi^2 \sqrt{\lambda}} \right\}.
\end{equation}

\subsubsection{Construction of the bound \boldmath $Y$ \unboldmath}

First note that $f_\bfk(\bar x)=0$ for any $\bfk\in \vect{I}_{2\bfm-1}$. Letting $y \bydef Af(\bar x)$, we get
$$
y_\bfk=
\begin{cases}
(A^{(\bfm)} f_{\bfFm}(\bar x))_\bfk,\quad & \bfk\in \bfFm\\
\mu_\bfk^{-1}f_\bfk(\bar x),& \bfk\in \vect{F}_{2\bfm-1}\setminus \bfFm\\
0, &{\rm otherwise}.
\end{cases}
$$
Hence, set
\begin{equation} \label{eq:Y}
Y=\sum_{\bfk\in \vect{F}_{2\bfm-1}}|y_\bfk|\nu^{|\bfk|},
\end{equation}
which is obtained via a finite computation.

\subsubsection{Construction of the bound  \boldmath$Z_0$ \unboldmath}

Denote by $B=I-AA^\dag$. Hence, set 
\begin{equation} \label{eq:Z0}
Z_0 \bydef || B||,
\end{equation}
which is computed using \eqref{eq:oper_norm}, and therefore is obtained by a finite computation. Note that the linear operator $B$ is finite dimensional, indeed $B$ acts non trivially only on the subspace  $X_{\bfFm}$. 

\subsubsection{Construction of the bounds \boldmath $Z_1$ \unboldmath and  \boldmath $Z_2$ \unboldmath}

We abuse notation and identify $u,v \in B(1) \in X$ and their counterparts $u_{sym}, v_{sym} \in X_{sym}$ with $(0,u_{sym})$, $(0,v_{sym})$. Similarly, when $\bar x$ appears in a convolution, it denotes the bi-infinite sequence $(\bar c_{0,0},\bar x_{sym})$. According to \eqref{eq:derh}, the action of $Df(\bar x+rv)$ on the element $ru$ is
\[
\Big(Df(\bar x+rv) u\Big)_{\bfk}=\mu_\bfk u_\bfk-2\Big(\big(c_{0,0}+{\rm sym}(\bar x+rv)\big)*\big({\rm sym}(u)\big)\Big)_\bfk,
\]
and thus
\[
\hspace{-.4cm}
\Big((Df(\bar x+rv)-A^\dag) u \Big)_{\bfk}
= \begin{cases}
 -2\Big(\big(c_{0,0}+{\rm sym}(\bar x)\big)*\big({\rm sym}(u_{\bfIm})\big)\Big)_\bfk -2({\rm sym}(v)*{\rm sym}(u))_\bfk r, &\bfk\in \bfFm\\
 -2\Big(\big(c_{0,0}+{\rm sym}(\bar x)\big)*\big({\rm sym}(u)\big)\Big)_\bfk -2({\rm sym}(v)*{\rm sym}(u))_\bfk r,  &\bfk\in \bfIm.
 \end{cases}
\]
Recall $\mathcal C_{\bfk,\bfj}(x)$ from \eqref{eq:Ckj} and  let $\overline C_{\bfk,\bfj}=\mathcal C_{\bfk,\bfj}(\bar x)$. For convenience, let us write explicitly $\overline C_{\bfk,\bfj}$ as
$$
\overline C_{\bfk,\bfj} 
= \begin{cases}
(\bar x_{\bfk-\bfj}+\bar x_{\bfk+\bfj}), &j_1=0 \\
(\bar x_{\bfk-\bfj}+\bar x_{\bfk+\bfj}+\bar x_{\bfk-(j_1,-j_2)}+\bar x_{\bfk+(j_1,-j_2)}), &j_1>0.
\end{cases}
$$
Since $\bar x_\bfk=0$ for $\bfk\not\in \bfFm$, for any $\bfj$ there is only a finite set of $\bfk$ so that $\overline C_{\bfk,\bfj}\neq 0$. For instance, if $\bfj \in \bfFm$ then $\overline C_{\bfk,\bfj}=0$ for any $\bfk\not\in \vect F_{2\vect m}$.
Define the linear operator $\Gamma: X\to X$ component-wise
$$
\Gamma_{(\bfk,\bfj)}
=\begin{cases}
0,\quad &(\bfk,\bfj)\in \bfFm \times \bfFm \\
-2\overline C_{\bfk,\bfj}, &{\rm otherwise}.
\end{cases}
$$
Then $(Df(\bar x+rv)-A^\dag) u = \Gamma u - 2 ( v_{sym}*u_{sym}) r$, and so
\[
A(Df(\bar x+rv)-A^\dag)u  =  (A\Gamma)u  - 2 A(v_{sym}*u_{sym})  r.
\]
Next, we look for $Z_1,Z_2$ satisfying $|| A\Gamma|| \le Z_1$, and $32||A|| \le Z_2$, since $\|v_{sym}*u_{sym}\|_\nu\leq 16$. \\

\noindent \underline{\bf Computation of \boldmath $Z_1$} \\ 

For any $(\bfs,\bfq)\in \mathcal Z^2_+\times \mathcal Z^2_+$, $(A\Gamma)_{(\bfs,\bfq)}=\sum_{\bfj\in \mathcal Z^2_+}A_{(\bfs,\bfj)}\Gamma_{(\bfj,\bfq)}$ and recalling \eqref{eq:oper_norm}, 
\begin{align*}
|| A\Gamma||&=\sup_{\bfq\in \mathcal Z^2_+}\frac{1}{\nu^{|\bfq|}}\sum_{\bfs\in \mathcal Z^2_+}|(A\Gamma)_{(\bfs,\bfq)}|\nu^{|\bfs|}\\
&=\max\left\{\sup _{\bfq\in \vect F_{2\vect m}}\frac{1}{\nu^{|\bfq|}}\sum_{\bfs\in \mathcal Z^2_+}|(A\Gamma)_{(\bfs,\bfq)}|\nu^{|\bfs|},\sup_{\bfq\in \vect I_{2\vect m}}\frac{1}{\nu^{|\bfq|}}\sum_{\bfs\in \mathcal Z^2_+}|(A\Gamma)_{(\bfs,\bfq)}|\nu^{|\bfs|}\right\}.
\end{align*}
Let us detail the two contributions separately. Denote 
\[
B(\bfq) \bydef \frac{1}{\nu^{|\bfq|}}\sum_{\bfs\in \mathcal Z^2_+}|(A\Gamma)_{(\bfs,\bfq)}|\nu^{|\bfs|}
=\frac{1}{\nu^{|\bfq|}}\sum_{\bfs\in \mathcal Z^2_+} \Big| \sum_{\bfj\in \mathcal Z^2_+}A_{(\bfs,\bfj)}\Gamma_{(\bfj,\bfq)} \Big| \nu^{|\bfs|} .
\]

{\bf Case 1 ($\bfq\in \vect F_{2\vect m}$):} This case is further decomposed in two sub cases: $\bfq\in \bfFm$ and $\bfq\in \vect F_{2\vect m}\setminus \bfFm$.
Assume that $\bfq\in \bfFm$. Since $\Gamma_{(\bfj,\bfq)}=0$ for $(\bfj,\bfq)\in \bfFm \times \bfFm$, we have 
\begin{align}
\nonumber
B(\bfq)&=\frac{1}{\nu^{|\bfq|}}\sum_{\bfs\in \mathcal Z^2_+}\left|\sum_{\bfj\in \bfIm}A_{(\bfs,\bfj)}\Gamma_{(\bfj,\bfq)})\right|\nu^{|\bfs|}\\
\nonumber
&=\frac{1}{\nu^{|\bfq|}}\sum_{\bfs\in \mathcal Z^2_+}\left|\sum_{\bfj\in \bfIm}\mu_{\bfs}^{-1}\delta_{\bfs,\bfj}(-2\overline C_{\bfj,\bfq})\right|\nu^{|\bfs|}\\
\nonumber
&=\frac{1}{\nu^{|\bfq|}}\sum_{\bfs\in \bfIm}\left|\mu_{\bfs}^{-1}(-2\overline C_{\bfs,\bfq})\right|\nu^{|\bfs|} \\
\nonumber
&=\frac{2}{\nu^{|\bfq|}}\sum_{\bfs\in \bfIm}|\mu_{\bfs}^{-1}(\overline C_{\bfs,\bfq})|\nu^{|\bfs|} \\
\label{eq:B(q) case 1}
&=\frac{2}{\nu^{|\bfq|}}\sum_{\bfs\in \vect F_{2\vect m}\setminus \bfFm}|\mu_{\bfs}^{-1}(\overline C_{\bfs,\bfq})|\nu^{|\bfs|},
\end{align}
which is a finite sum, since for any $s\in \vect I_{2\vect m}$, we get that $\overline C_{\bfs,\bfq}=0$ for all $\bfq\in \bfFm$.

Assume now that $\bfq\in \vect F_{2\vect m}\setminus \bfFm$. In this case,
\begin{align} 
\nonumber
B(\bfq)&=\frac{1}{\nu^{|\bfq|}}\left[ \sum_{\bfs\in \bfFm} \Big| \sum_{\bfj\in \mathcal Z^2_+}A_{(\bfs,\bfj)}\Gamma_{(\bfj,\bfq)} \Big| \nu^{|\bfs|}  +\sum_{\bfs\in \bfIm} \Big| \sum_{\bfj\in \mathcal Z^2_+}A_{(\bfs,\bfj)}\Gamma_{(\bfj,\bfq)} \Big| \nu^{|\bfs|} \right]\\
\nonumber
&=\frac{1}{\nu^{|\bfq|}}\left[ \sum_{\bfs\in \bfFm} \Big| \sum_{\bfj\in \bfFm}A_{(\bfs,\bfj)}(-2\overline C_{\bfj,\bfq}) \Big| \nu^{|\bfs|}  
+\sum_{\bfs\in \bfIm} \Big| \mu_\bfs^{-1}\Gamma_{(\bfs,\bfq)} \Big| \nu^{|\bfs|} \right]\\
\nonumber
&=\frac{2}{\nu^{|\bfq|}}\left[ \sum_{\bfs\in \bfFm} \Big| \sum_{\bfj\in \bfFm}A_{(\bfs,\bfj)}\overline C_{\bfj,\bfq} \Big| \nu^{|\bfs|}  +\sum_{\bfs\in \bfIm} \Big| \mu_\bfs^{-1}\overline C_{\bfs,\bfq}\Big|\nu^{|\bfs|} \right] \\
\label{eq:B(q) case 2}
&=\frac{2}{\nu^{|\bfq|}}\left[ \sum_{\bfs\in \bfFm} \Big| \sum_{\bfj\in \bfFm}A_{(\bfs,\bfj)}\overline C_{\bfj,\bfq} \Big| \nu^{|\bfs|}  +\sum_{\bfs\in \vect F_{3\vect m}\setminus \bfFm} \Big| \mu_\bfs^{-1}\overline C_{\bfs,\bfq}\Big|\nu^{|\bfs|} \right] .
\end{align}

{\bf Case 2 ($\bfq\in \vect I_{2\vect m}$):} In this case, the only possibility for $\bfj$ giving that $\Gamma_{(\bfj,\bfq)}\neq 0$  is  that $\bfj\in \bfIm$.  For $\bfj\in \bfIm$ the operator  $A$ is {\it diagonal }, therefore non zero contributions to the sum are given only when $\bfs=\bfj\in \bfIm$. These considerations imply that 
$$
B(\bfq)=\frac{1}{\nu^{|\bfq|}}\sum_{\bfs\in \bfIm} \Big| \mu_{\bfs}^{-1}\Gamma_{(\bfs,\bfq)} \Big|\nu^{|\bfs|} 
= \frac{1}{\nu^{|\bfq|}}\sum_{\bfs\in \bfIm} \Big| \mu_{\bfs}^{-1}(-2\overline C_{\bfs,\bfq}) \Big| \nu^{|\bfs|}
= \frac{2}{\nu^{|\bfq|}}\sum_{\bfs\in \bfIm} \Big| \mu_{\bfs}^{-1}\overline C_{\bfs,\bfq} \Big| \nu^{|\bfs|}.
$$
We remind that the sequence $\bar x_\bfk$ is symmetric, that is $\bar x_{\pm k_1,\pm k_2}=\bar x_{k_1,k_2}$, and that $\bar x_\bfk=0$ for any $\bfk\in \bfIm$. Thus, denoting by
\[
\bfFm^{\pm}=\{ \bfk=(k_1,k_2)\in \Z^2: |k_1|<m_1, 0<|k_2|<m_2\},\quad \bfIm^\pm=\Z^2\setminus \bfFm^\pm
\]
we simply have that $\bar x_\bfk=0$ for any $\bfk\in \bfIm^\pm$.
 
In the situation under consideration, the only possible non zero contribution to $\overline C_{\bfs,\bfq}$ is given by $\bar x_{\bfs-\bfq}$. Indeed all the other combinations $\bfs+\bfq$ and $\bfs\pm(q_1,-q_2)$ give an index outside $\bfFm^\pm$. Thus
\[
B(\bfq)=\frac{2}{\nu^{|\bfq|}}\sum_{\bfs\in \bfIm}| \mu_{\bfs}^{-1} \bar x_{\bfs-\bfq} |\nu^{|\bfs|}\leq \frac{2}{\nu^{|\bfq|}}\sum_{\bfp\in \bfFm^\pm}| \mu_{\bfp+\bfq}^{-1}\bar x_{\bfp}|\nu^{|\bfp+\bfq|} \leq 2\sum_{\bfp\in \bfFm^\pm}| \mu_{\bfp+\bfq}^{-1}\bar x_{\bfp}|\nu^{|\bfp|}.
\]
%
Denote
\[
\mathcal B(\bfq) \bydef 2\sum_{\bfp\in \bfFm^\pm}| \mu_{\bfp+\bfq}^{-1}\bar x_{\bfp}|\nu^{|\bfp|}.
\]
From the previous relation, it follows that 
\[
\sup_{\bfq\in \vect I_{2\vect m}}B(\bfq)\leq \sup_{\bfq\in \vect I_{2\vect m}}\mathcal B(\bfq).
\]

Since  $\mu_\bfs^{-1}$ is {\it decreasing}, it is enough to restrict the computation of $\mathcal B(\bfq)$ to the finite set $\bfq\in \vect I_{2\vect m}\cap \vect F_{3\vect m}$, as stated in the following Lemma. Before presenting the result, let us introduce the sets $\vect R_n$, named the {\em rings}, as 
\begin{equation}\label{eq:rings}
\vect R_n\bydef \vect F_{(n+1)\vect m}\setminus \vect F_{n\vect m}=\vect I_{n\vect m}\cap \vect F_{(n+1)\vect m}.
\end{equation}

\begin{lemma} \label{lem:technical_lemma1}
If $\bfm=(m_1,m_2)$ satisfies the condition \eqref{eq:cond_m}, then
\[
\sup_{\bfq\in \vect I_{2\vect m}}\mathcal B(\bfq)=\max_{\bfq\in \vect R_{2}}\mathcal B(\bfq).
\]
\end{lemma}

\begin{proof}
First note that $\mu_{\bfp+\bfq}^{-1}$ is positive for $\bfq\in \vect I_{2\vect m}$ and $\bfp\in \bfFm^\pm$.
Therefore 
\[
\mathcal B(\bfq)=2\sum_{\bfp\in \bfFm^\pm} \mu_{\bfp+\bfq}^{-1}|\bar x_{\bfp}|\nu^{|\bfp|}.
\]
The proof follows by showing that for any choice of $\bfq\in \vect I_{3\vect m}$ there exists  $\hat\bfq\in \vect R_2$
such that
\begin{equation}
\label{eq:inequality_mu}
\mu_{\bfp+\hat\bfq} \leq \mu_{\bfp+\bfq}, \quad \text{for all } \bfp\in \bfFm^\pm.
\end{equation}
Indeed, from \eqref{eq:inequality_mu} we conclude that  $\mathcal B(\hat \bfq)\geq \mathcal B(\bfq)$,
and then the result follows from $\vect{I}_{2 \vect {m}} =  \vect{I}_{3 \vect {m}} \cup \vect{R}_{2}$.

In order to show \eqref{eq:inequality_mu} let us consider the \emph{ring}-like decomposition
\[
\vect I_{3\vect m}=\bigcup_{n\geq 3}\vect R_{n}.
\]
We prove the statement for $\bfq \in \vect R_3$ and then by induction we extend it to $\vect R_n$,
$n\geq 4$. We introduce the two disjoints set of indices
\begin{align*}
A_+ & \bydef \{\bfq=(q_1,q_2) \in \mathcal Z^2_+ : q_1\in [3m_1, 4m_1), q_2\in(0, 4m_2)\} \subset \vect R_3
\\
B_+ & \bydef \{\bfq=(q_1,q_2) \in \mathcal Z^2_+ : q_1\in [0, 3m_1), q_2\in[3m_2, 4m_2)\} \subset \vect R_3
\end{align*}
so that 
\[
\vect R_3  = \vect F_{4\vect m}\setminus \vect F_{3\vect m} = A_+ \cup B_+.
\]
Let $\bfq \in \vect R_3$. Suppose first that $\bfq\in B_+$, and define $\hat \bfq=(q_1,q_2-m_2) \in \vect R_2$.
For any $\bfp\in \bfFm^\pm$ 
\[
\begin{aligned}
\mu_{\bfp+\bfq}-\mu_{\bfp+\hat \bfq} & =
\frac{L^2}{4\pi^2} (p_1+q_1)^2 \frac{(p_2+q_2-m_2)^2-(p_2+q_2)^2}{(p_2+q_2-m_2)^2(p_2+q_2)^2} +
4\pi^2 \lambda \Big((p_2+q_2)^2 - (p_2+q_2-m_2)^2\Big) \\
& = \frac{L^2}{4\pi^2} (p_1+q_1)^2 \frac{m_2^2 - 2m_2(p_2+q_2)}{(p_2+q_2-m_2)^2(p_2+q_2)^2} +
4\pi^2\lambda(2m_2 (p_2+q_2) - m_2^2) \\
& = m_2 (2(p_2+q_2)-m_2) \left[ 4\pi^2\lambda-\frac{L^2}{4\pi^2}
\frac{(p_1+q_1)^2}{(p_2+q_2-m_2)^2(p_2+q_2)^2}\right].
\end{aligned}
\]
From \eqref{eq:cond_m} we get that $m_2^2  \geq \frac{L^2}{4\pi^4\lambda}$ and $m_2 \geq m_1$,
and so
\begin{equation}
\label{eq:technical_ineq1}
4 \pi^2 \lambda - \frac{L^2}{4\pi^2} \frac{(p_1+q_1)^2}{(p_2+q_2-m_2)^2(p_2+q_2)^2} \geq
4 \pi^2 \lambda - \frac{L^2}{4\pi^2} \frac{(4m_1)^2}{(m_2)^2(2m_2)^2} \geq 0.
\end{equation}
Combining \eqref{eq:technical_ineq1} with the fact that $2(p_2+q_2) - m_2 > 0$, it follows that
$\mu_{\bfp+\bfq}-\mu_{\bfp+\hat \bfq} \geq 0$. 

Suppose now that $\bfq\in A_+$, and define $\tilde \bfq=(q_1-m_1,q_2)$. Then we have
\begin{align*}
\mu_{\bfp+\tilde \bfq} & = \frac{L^2}{4\pi^2} \frac{(p_1+\tilde q_1)^2}{(p_2+\tilde q_2)^2} +
4 \pi^2\lambda (p_2+\tilde q_2)^2-1 \\
& = \frac{L^2}{4\pi^2} \frac{(p_1+q_1-m_1)^2}{(p_2+ q_2)^2} + 
4 \pi^2\lambda (p_2+ q_2)^2-1 < \mu_{\bfp+\bfq}.
\end{align*}
If $\tilde \bfq \in \vect R_2$, define $\hat \bfq=\tilde \bfq $, and so $\mu_{\bfp+\bfq} > \mu_{\bfp+\hat\bfq}$.
Otherwise if $\tilde \bfq \not \in \vect R_2$ then necessarily $\tilde \bfq\in B_+$.
Thus, define $\hat \bfq = (\tilde q_1, \tilde q_2 - m_2) = (q_1 - m_1, q_2 - m_2) \in \vect R_2$. Proceeding
as above, we get that $\mu_{\bfp+\bfq} > \mu_{\bfp+\tilde\bfq} \geq \mu_{\bfp+\hat\bfq}$.
We then conclude that for any $\bfq\in \vect R_3$ the statement holds, that is, for every $\bfq \in \vect R_3$
there exists $\hat\bfq\in \vect R_2$ such that $\mu_{\bfp+\hat\bfq} \leq \mu_{\bfp+\bfq}$, for all $\bfp\in \bfFm^\pm$.

By induction on $n$ it follows that for any $\bfq\in \vect R_n$ there exists a $\hat \bfq\in \vect R_{n-1}$ so that
$\mu_{\bfp+\hat\bfq} \leq \mu_{\bfp+\bfq}$ for any $\bfp\in \bfFm^\pm$ provided that
$m_2^2 \geq \frac{L^2}{16\pi^4 \lambda}\frac{(n+1)^2}{(n-2)^2(n-1)^2}$, which ensures that an inequality
analogous to \eqref{eq:technical_ineq1} holds. Since the function $g(n) \bydef \frac{(n+1)^2}{(n-2)^2(n-1)^2}$
is decreasing in $n$, condition \eqref{eq:cond_m} guarantees that
$m_2^2 \geq \frac{L^2}{4 \pi^4\lambda}g(n) > \frac{L^2}{16\pi^4\lambda}g(n) $ for any $n\geq 4$.
\end{proof}

Set
\[
B^{(1)} \bydef \max_{\bfq\in \vect F_{2\vect m}} B(\bfq), \qquad
B^{(2)} \bydef \max_{\bfq\in \vect F_{2\vect m}\setminus \bfFm} B(\bfq) \quad \text{and} \quad
B^{(3)} \bydef \max_{\bfq\in \vect R_{2}}\mathcal B(\bfq).
\]
Using formulas \eqref{eq:B(q) case 1} and \eqref{eq:B(q) case 2} and the result of Lemma~\ref{lem:technical_lemma1}, we set
\begin{equation} \label{eq:Z1}
Z_1 
=\max \left\{B^{(1)}, B^{(2)},B^{(3)} \right\},
\end{equation}
which is obtained via a finite computation.  \\

\noindent \underline{\bf Computation of \boldmath $Z_2$\unboldmath} \\ 

Recalling the definition of the rings \eqref{eq:rings} and   arguing as in the proof of Lemma~\ref{lem:technical_lemma1},
it is easy to prove the following.
\begin{lemma} \label{lem:technical_lemma2}
If $\bfm=(m_1,m_2)$ satisfies \eqref{eq:cond_m}, then 
\[
||A||\leq \max\{|| A^{(\bfm)}||, \max_{\bfk\in \vect R_1}\mu_\bfk^{-1}\}.
\]
\end{lemma}

\begin{proof}
From the definition of $A$ and formula  \eqref{eq:oper_norm} it follows that
\[
||A||\leq \max\{ ||A^{(\bfm)}||, \sup_{\bfk\in \bfIm}\mu_\bfk^{-1}\}.
\]
First we show that for any $\bfk\in \vect R_2$ there exists a $\widehat \bfk\in \vect R_1$ so that $\mu_\bfk \geq \mu_{\widehat \bfk}$. 
Following the same arguments as in the proof of Lemma~\ref{lem:technical_lemma1}, setting $p_1=p_2=0$ and with
the obvious adaptation, it follows that $\mu_\bfk \geq \mu_{\widehat \bfk}$ provided that 
\[
4\pi^2\lambda-\frac{L^2}{4\pi^2}\frac{k_1^2}{(k_2-m_2)^2k_2^2} \geq 0, \quad
\text{for all } (k_1,k_2)\in B_+ \bydef [0,2m_1)\times [2m_2,3m_2). 
\]
Now for any $(k_1,k_2)\in B_+$ we have
\[
4\pi^2\lambda-\frac{L^2}{4\pi^2}\frac{k_1^2}{(k_2-m_2)^2k_2^2} \geq
4\pi^2\lambda-\frac{L^2}{4\pi^2} \frac{m_1^2}{m_2^4} \geq 0
\]
whenever $m_2^2 \geq \frac{L^2}{16\pi^4 \lambda}$ and $m_2 \geq m_1$, which is guaranteed
by assumption \eqref{eq:cond_m}. Then, by induction, since $g(n)=1/n$ is decreasing, for any
$\bfk\in \vect R_n$, there exists a sequence $\widehat{\bfk}_1,\dots, \widehat{\bfk}_{n-1}$, with
$\widehat\bfk_{i}\in \vect R_{n-i}$, such that
\[
\mu_{\bfk} \geq \mu_{\widehat \bfk_1} \geq \dots \geq \mu_{\widehat\bfk_{n-1}}.
\]
\end{proof}
Therefore, using Lemma~\ref{lem:technical_lemma2}, we set
\begin{equation} \label{eq:Z2} 
Z_2  \bydef 32 \max\left\{|| A^{(\bfm)}||, \max_{\bfk\in \vect R_1}\mu_\bfk^{-1}\right\},
\end{equation}
which is obtained via a finite computation.

Combining the bounds $Y$, $Z_0$, $Z_1$ and $Z_2$ given respectively by \eqref{eq:Y}, \eqref{eq:Z0}, \eqref{eq:Z1} and \eqref{eq:Z2}, we have explicitly constructed the radii polynomial $p(r)$ as defined in \eqref{eq:radii_polynomial}.
  
\section{Results}
\label{sec:results}

In this section we present several computer-assisted proofs of existence of periodic orbits of
\eqref{eq:boussinesq}. For the results presented in this section we started with four numerical approximations
of periodic orbits corresponding to a small value of $\lambda$ and applied a numerical continuation algorithm
to each one of these solutions to get several numerical solutions along a branch of solutions. These numerically
computed branches are plotted in Figures~\ref{fig:solns_branch1}-\ref{fig:solns_branch4}. We selected three
numerical solutions along each branch and applied our rigorous method to prove the existence of a true
periodic orbit to the Boussinesq equation \eqref{eq:boussinesq} close to these numerical solutions.
The points along each branch for which we produced a computer-assisted proof, as well as a plot of each
numerical solution, are presented in Figures~\ref{fig:solns_branch1}-\ref{fig:solns_branch4}. The three points
in each branch were selected using the following criteria: The first point was selected at the beginning of the
branch in order for the proof to succeed with small values of $m_1$ and $m_2$, and hence for the verification
code to run fast; the second point was chosen with the aim of getting a small value of $r$ in the proof while
having $\lambda$ not so small and $m_1$ and $m_2$ not so large; the third point was selected with the aim
of maximizing $\lambda$ while not having $m_1$ and $m_2$ too large. The values of these parameters, as
well as the running time for the proof, are presented in Figures~\ref{fig:solns_branch1}-\ref{fig:solns_branch4}
for each one of the solutions that were proved to exist. All the numerical data, as well as the code to perform
the proofs, are presented in \cite{webpage_boussinesq}.

For all the proofs presented in this paper we used $L=2\pi$  ($1$-periodic in space), $m_1 = m_2$, and $\nu = 1.01$
as the decay rate. The proofs are made rigorous by using the interval arithmetic package \textsf{INTLAB}~\cite{Ru99a}.

\begin{theorem} \label{thm:solns_branch1}
Let $\lambda \in \{ 0.1446, 0.2346, 1.0846 \}$ and consider the corresponding numerical approximation
$\bar{u}(t,y)$ depicted Figure~\ref{fig:solns_branch1}. Then there exists a classical periodic solution $u(t,y)$
of \eqref{eq:boussinesq} having $C^0$- and $L^2$-error bounds both equal to $4r$, where the value of $r$
is presented in Figure~\ref{fig:solns_branch1}.
\end{theorem}

\begin{proof}
Given the three numerical approximations at $\lambda \in \{ 0.1446, 0.2346, 1.0846 \}$,
for each of the corresponding numerical approximations, 
the MATLAB script {\tt script\_proof\_theorem\_1.m} computes the coefficients 
$Y$, $Z_0$, $Z_1$ and $Z_2$ given respectively by \eqref{eq:Y}, \eqref{eq:Z0}, \eqref{eq:Z1} and \eqref{eq:Z2} and it verifies with \textsf{INTLAB} (interval arithmetic in \textsf{MATLAB}) the existence of an interval $\cI=(r_{\rm min},r_{\rm max})$ such that for each $r \in \cI$, $p(r)<0$, with $p(r)$ the radii polynomial as defined in \eqref{eq:radii_polynomial}.
By Lemma~\ref{lem:rad_poly}, there exists a unique $\tx \in B_r(\bx)$ such that $f(\tx)=0$, with $f$ given component-wise in \eqref{eq:f_k}. By construction, this corresponds to a periodic orbits of the Boussinesq equation \eqref{eq:boussinesq}.
\end{proof}

\begin{theorem}
\label{thm:solns_branch2}
Let $\lambda \in \{ 0.1596, 0.2796, 0.2846 \}$ and consider the corresponding numerical approximation
$\bar{u}(t,y)$ depicted Figure~\ref{fig:solns_branch2}. Then there exists a classical periodic solution $u(t,y)$
of \eqref{eq:boussinesq} having $C^0$- and $L^2$-error bounds both equal to $4r$, where the value of $r$
is presented in Figure~\ref{fig:solns_branch2}.
\end{theorem}

\begin{proof}
The proof is similar as the proof of Theorem~\ref{thm:solns_branch1}, and is done by running the
MATLAB script {\tt script\_proof\_theorem\_2.m}.
\end{proof}

\begin{theorem}
\label{thm:solns_branch3}
Let $\lambda \in \{ 0.1846, 0.2746, 0.2796 \}$ and consider the corresponding numerical approximation
$\bar{u}(t,y)$ depicted Figure~\ref{fig:solns_branch3}. Then there exists a classical periodic solution $u(t,y)$
of \eqref{eq:boussinesq} having $C^0$- and $L^2$-error bounds both equal to $4r$, where the value of $r$
is presented in Figure~\ref{fig:solns_branch3}.
\end{theorem}

\begin{proof}
The proof is similar as the proof of Theorem~\ref{thm:solns_branch1}, and is done by running the
MATLAB script {\tt script\_proof\_theorem\_3.m}.
\end{proof}

\begin{theorem}
\label{thm:solns_branch4}
Let $\lambda \in \{ 0.1356, 0.1446, 0.2146 \}$ and consider the corresponding numerical approximation
$\bar{u}(t,y)$ depicted Figure~\ref{fig:solns_branch4}. Then there exists a classical periodic solution $u(t,y)$
of \eqref{eq:boussinesq} having $C^0$- and $L^2$-error bounds both equal to $4r$, where the value of $r$
is presented in Figure~\ref{fig:solns_branch4}.
\end{theorem}

\begin{proof}
The proof is similar as the proof of Theorem~\ref{thm:solns_branch1}, and is done by running the
MATLAB script {\tt script\_proof\_theorem\_4.m}.
\end{proof}

All computer-assisted proofs of the above theorems were made on an iMac with a 3.4GHz processor and 16GB of memory.
The statements involving the $C^0$- and $L^2$-error bounds in the theorems are justifies in the next final short section.

\begin{figure}[!htbp]
  \centering
  \begin{subfigure}[t]{0.5\textwidth}
      \centering
      \includegraphics[width=0.95\linewidth]{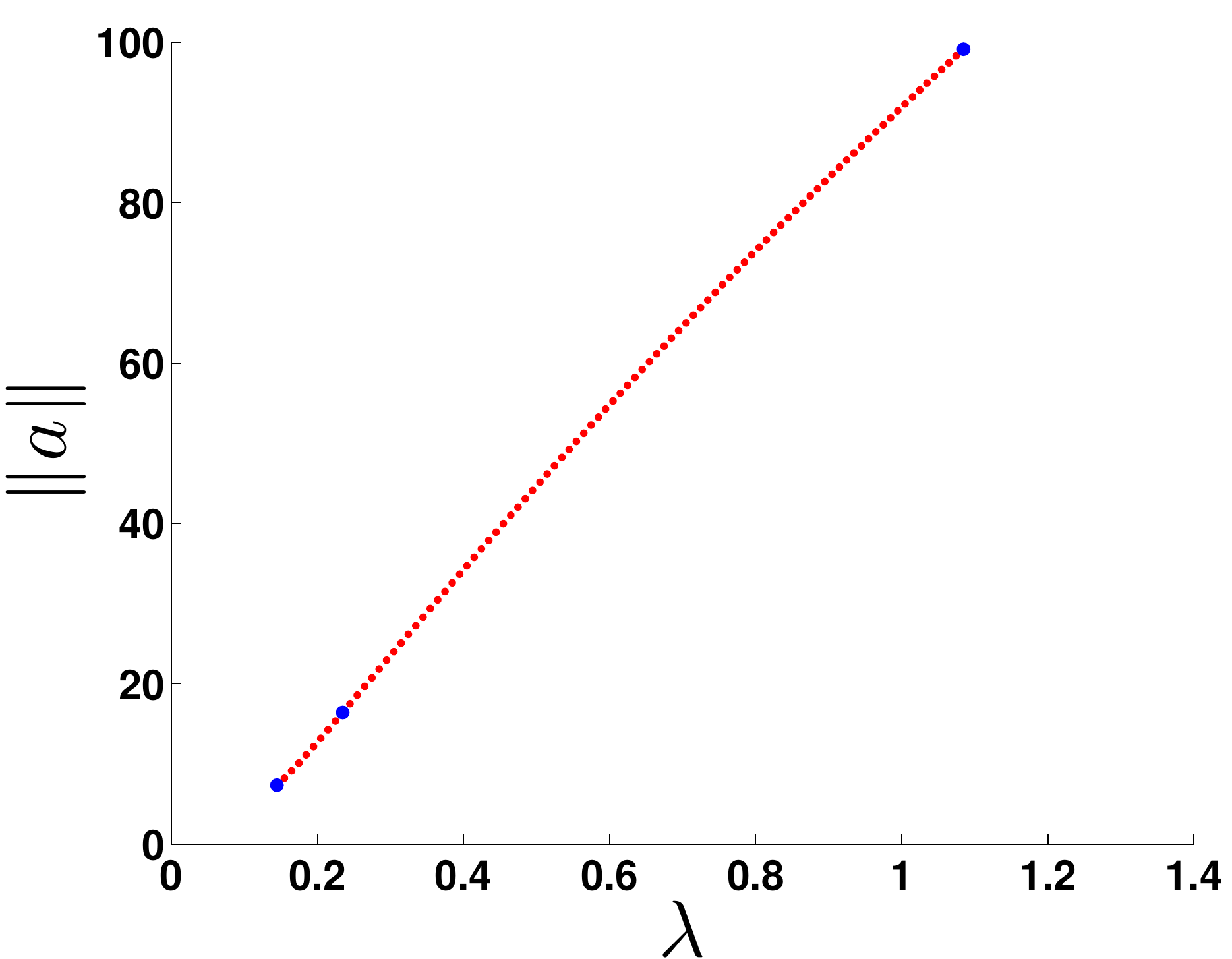}
      \caption{Branch of solutions.}
  \end{subfigure}%
  \begin{subfigure}[t]{0.5\textwidth}
      \centering
      \includegraphics[width=0.95\linewidth]{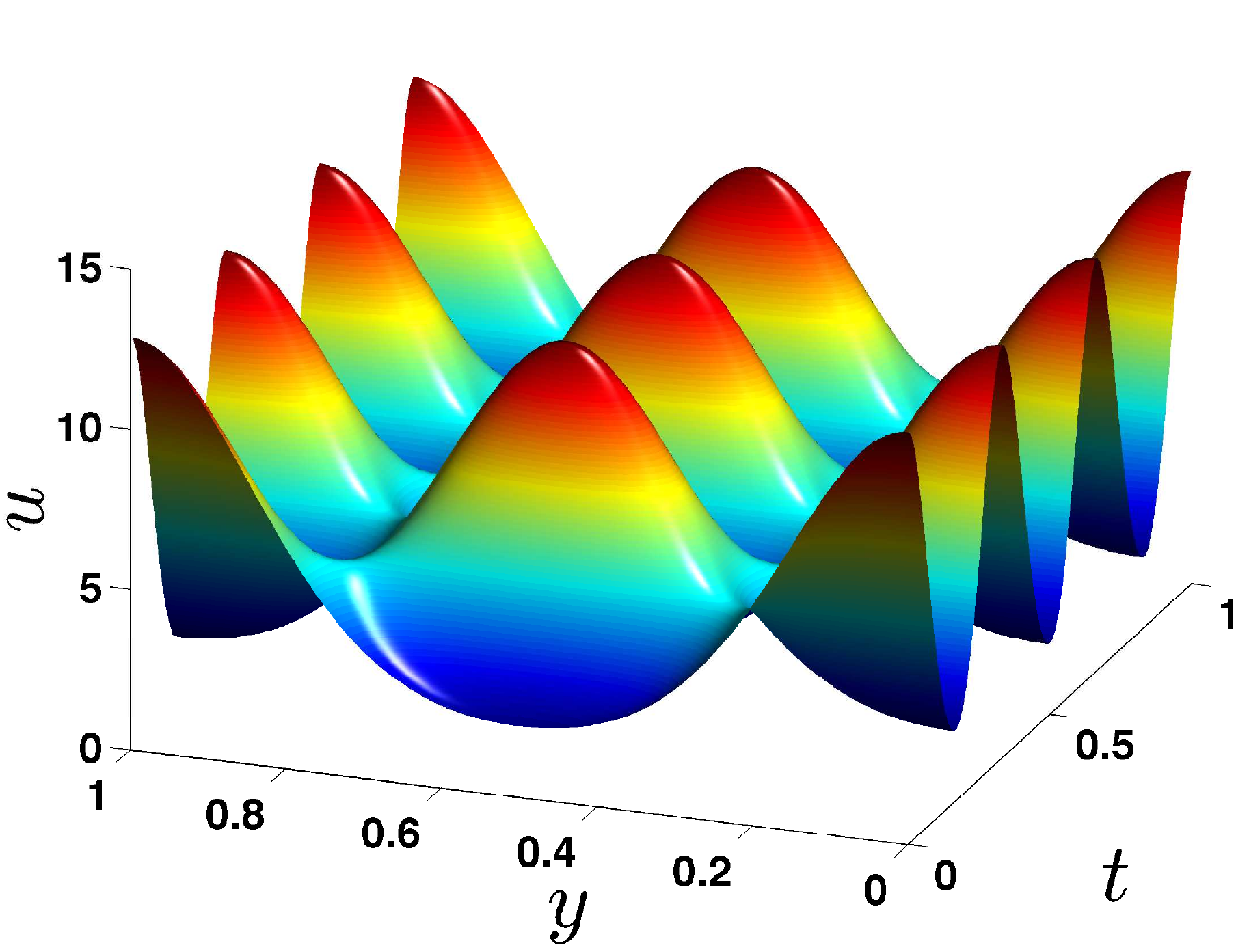}
      \caption{$\lambda = 0.1446$}
  \end{subfigure}
  \begin{subfigure}[t]{0.5\textwidth}
      \centering
      \includegraphics[width=0.95\linewidth]{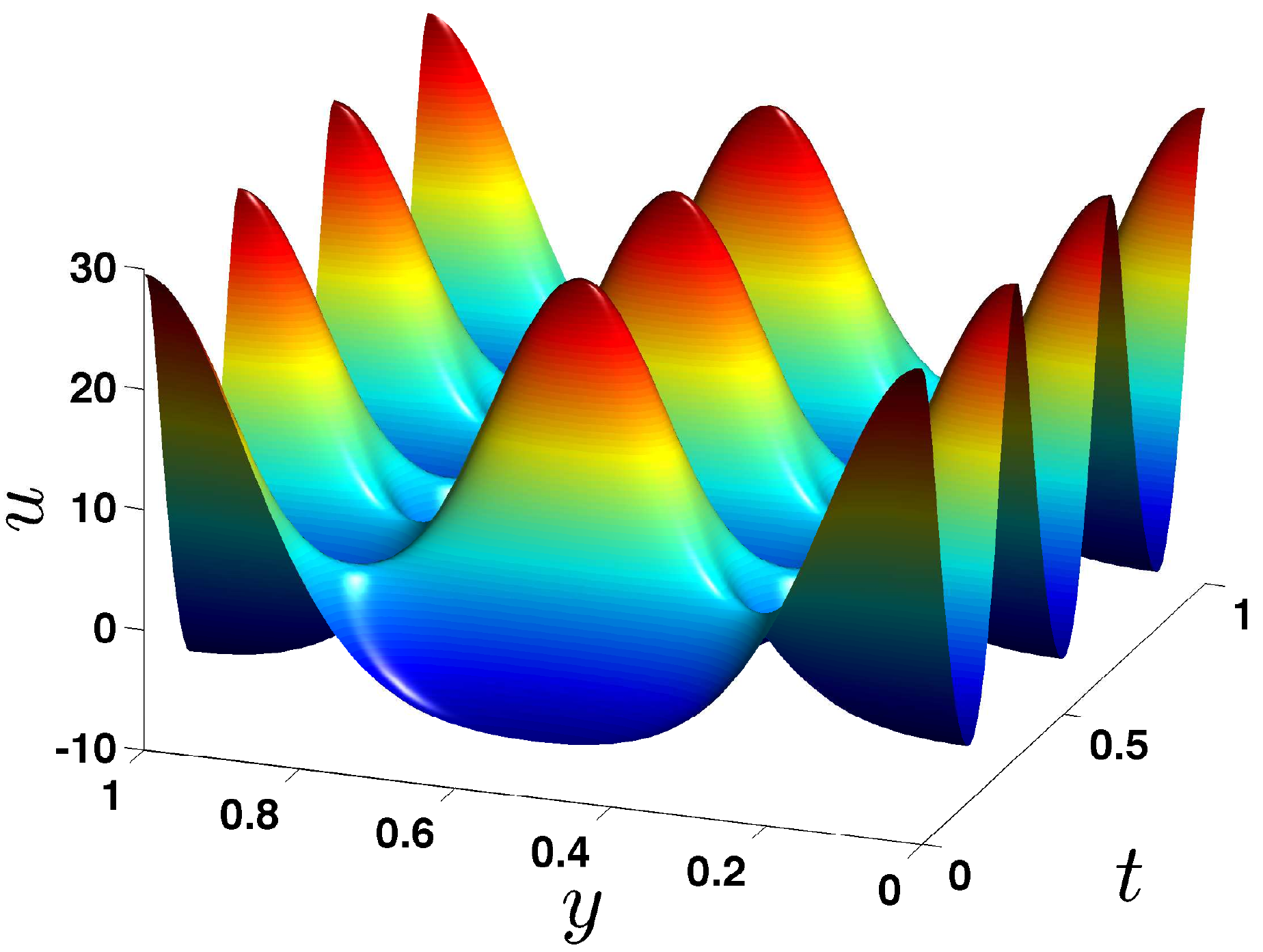}
      \caption{$\lambda = 0.2346$}
  \end{subfigure}%
  \begin{subfigure}[t]{0.5\textwidth}
      \centering
      \includegraphics[width=0.95\linewidth]{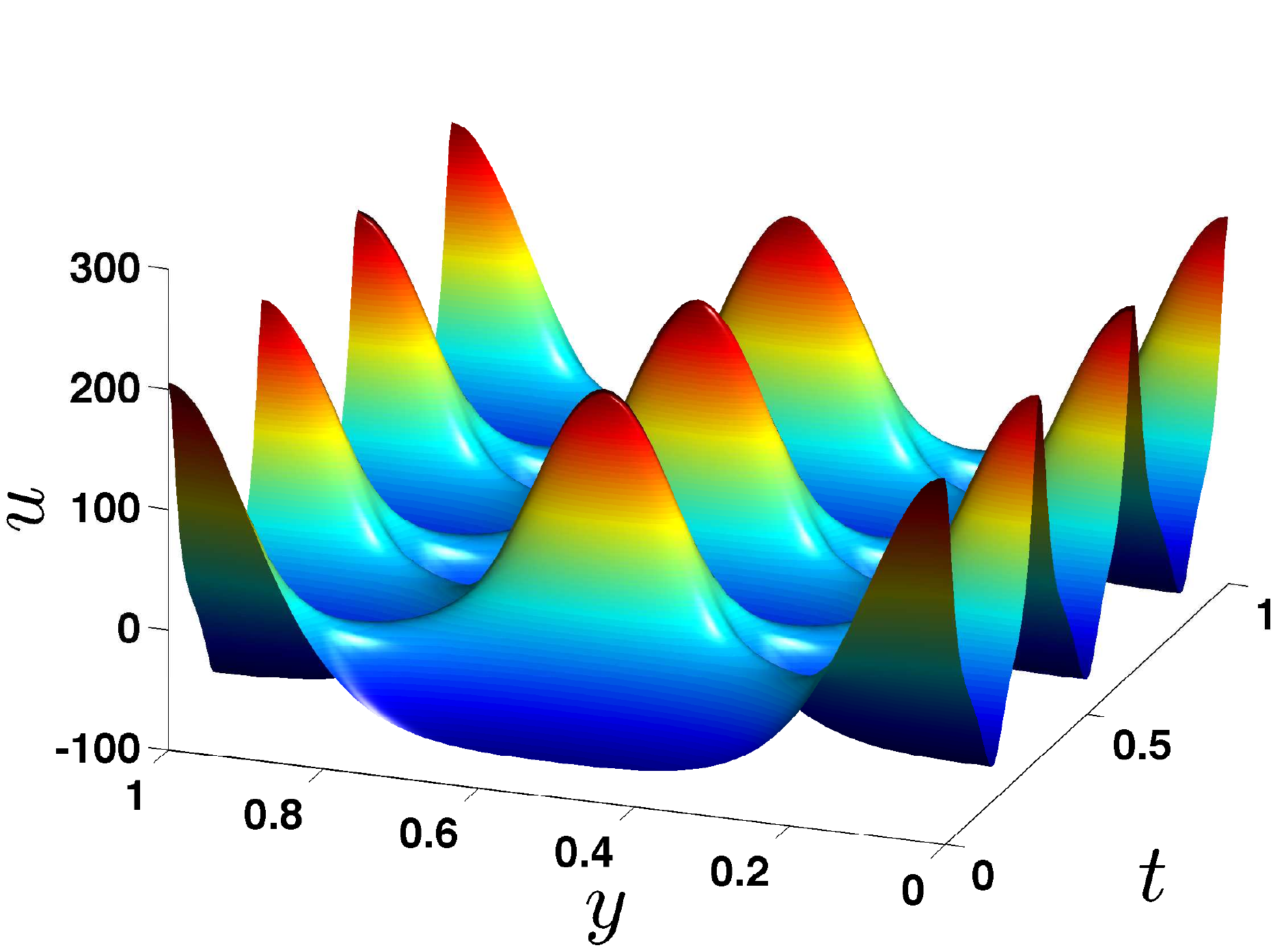}
      \caption{$\lambda = 1.0846$}
  \end{subfigure}%
  \caption{In (a) we plot the branch of numerically computed solutions from which
  we proved, in Theorem~\ref{thm:solns_branch1}, the existence of three periodic
  orbits. These solutions correspond to the larger dots on the curve, and are
  plotted in (b), (c), and (d), where the values of $\lambda$ to which they correspond are
  indicated. The values of the projection dimensions $m_1 = m_2$, the radius $r$, and the
  running time for the proofs are: (b) $m_1 = m_2 = 35$, $r = 1.07191 \times 10^{-11}$,
  running time $50.06$ seconds; (c) $m_1 = m_2 = 61$, $r = 1.45275 \times 10^{-11}$,
  running time $224.42$ seconds; (d) $m_1 = m_2 = 61$, $r = 1.09053 \times 10^{-3}$,
  running time $223.79$ seconds.}
  \label{fig:solns_branch1}
\end{figure}

\begin{figure}[!htbp]
  \centering
  \begin{subfigure}[t]{0.5\textwidth}
      \centering
      \includegraphics[width=0.95\linewidth]{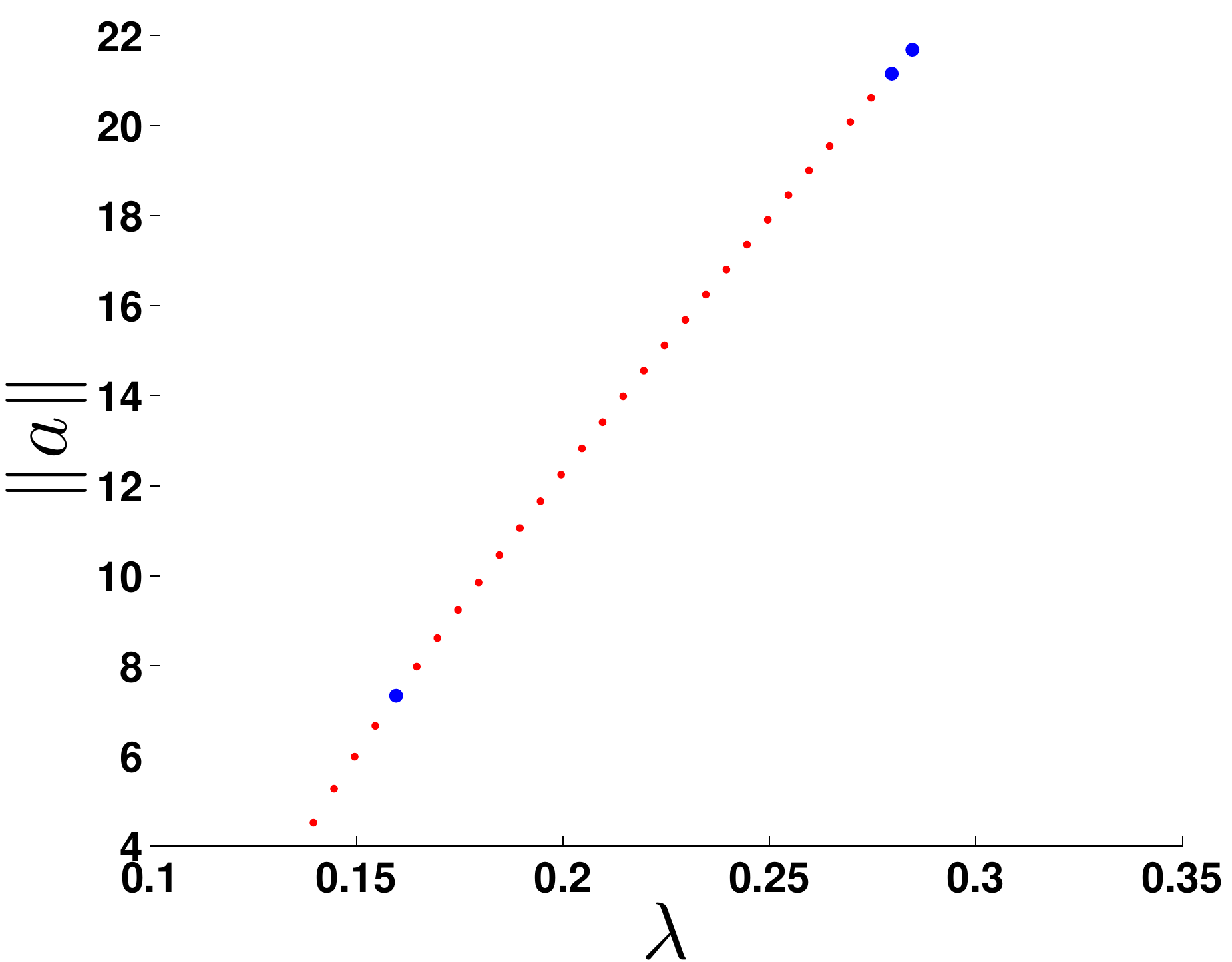}
      \caption{Branch of solutions.}
  \end{subfigure}%
  \begin{subfigure}[t]{0.5\textwidth}
      \centering
      \includegraphics[width=0.95\linewidth]{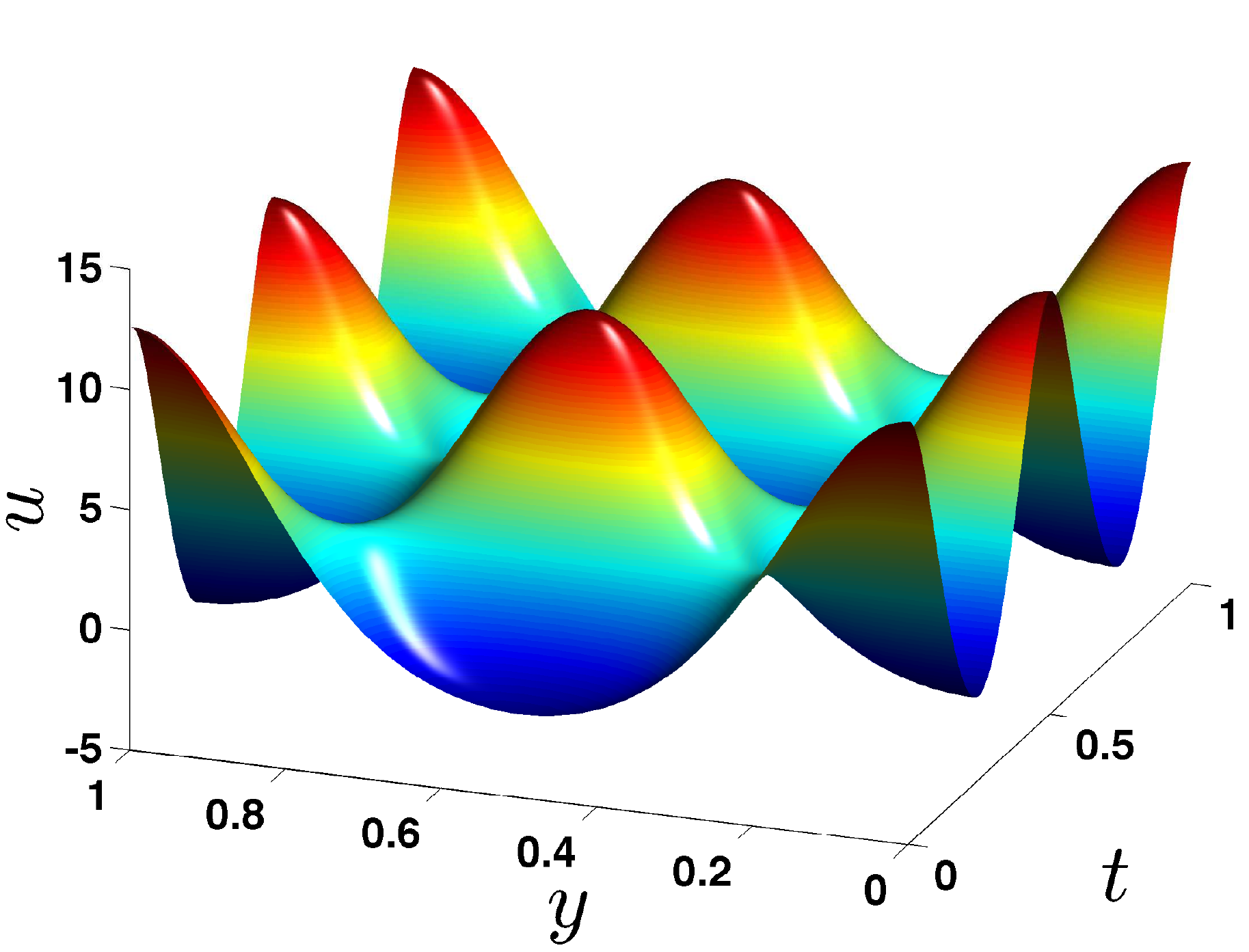}
      \caption{$\lambda = 0.1596$}
  \end{subfigure}
  \begin{subfigure}[t]{0.5\textwidth}
      \centering
      \includegraphics[width=0.95\linewidth]{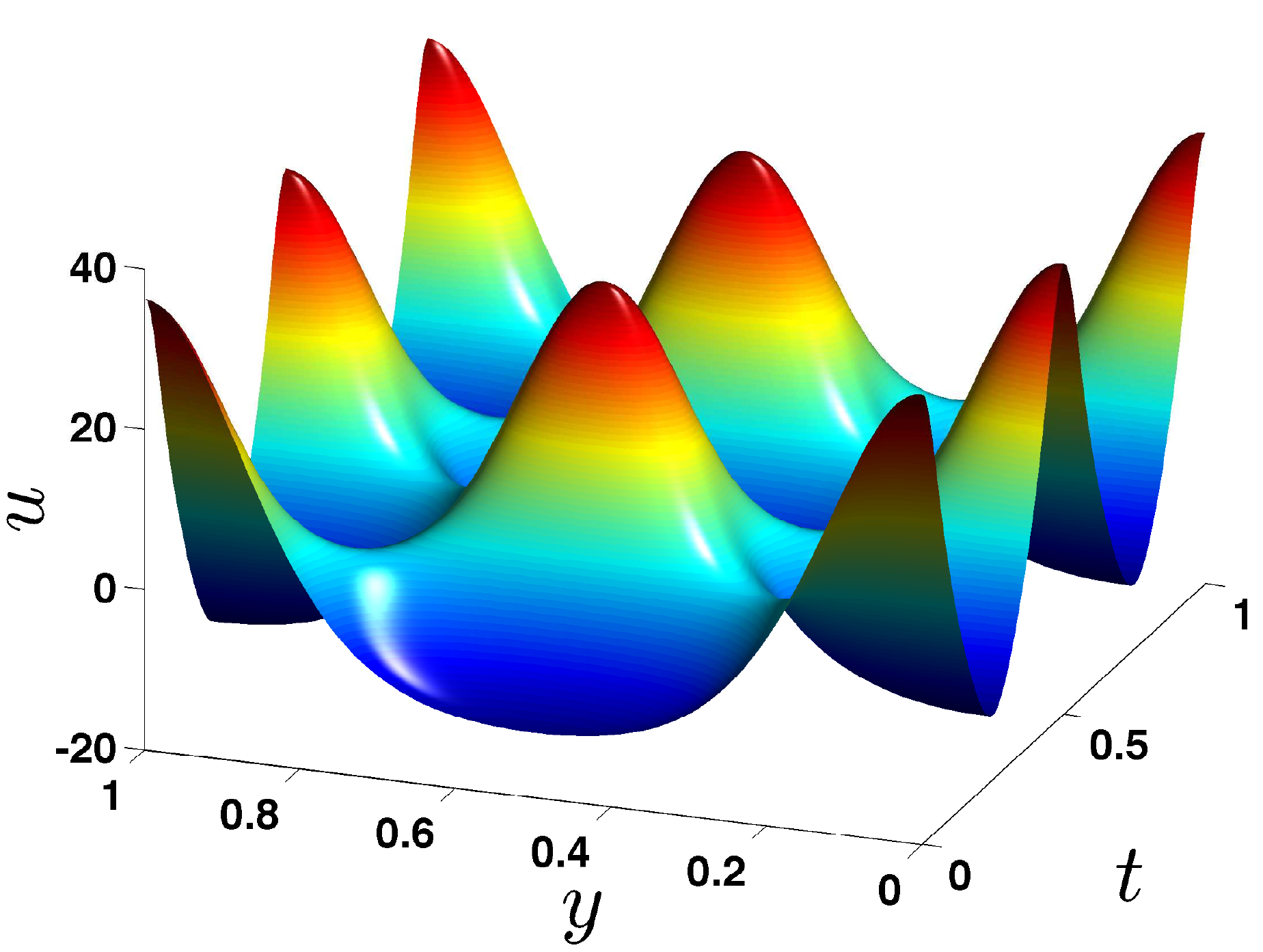}
      \caption{$\lambda = 0.2796$}
  \end{subfigure}%
  \begin{subfigure}[t]{0.5\textwidth}
      \centering
      \includegraphics[width=0.95\linewidth]{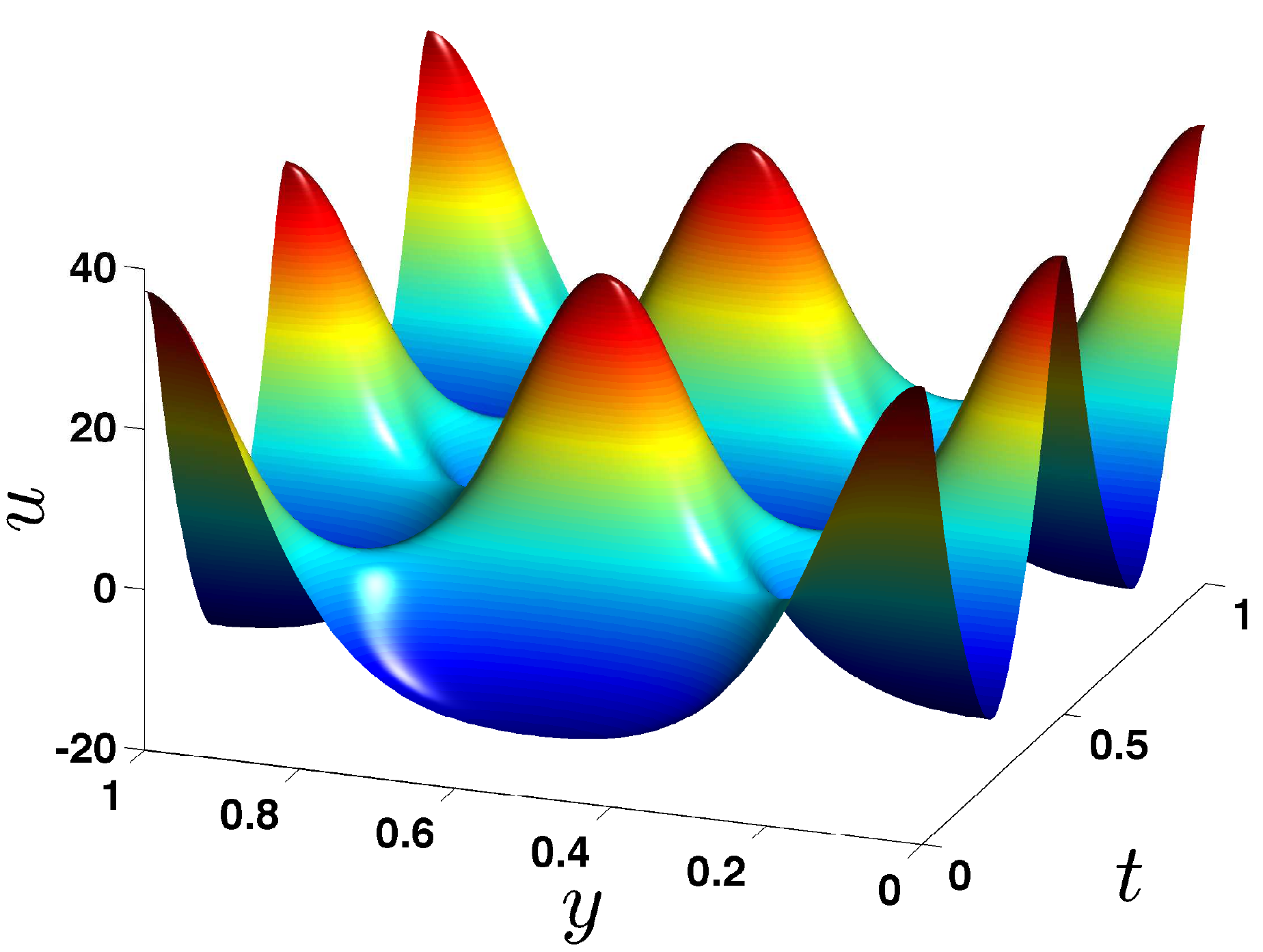}
      \caption{$\lambda = 0.2846$}
  \end{subfigure}%
  \caption{In (a) we plot the branch of numerically computed solutions from which
  we proved, in Theorem~\ref{thm:solns_branch2}, the existence of three periodic
  orbits. These solutions correspond to the larger dots on the curve, and are
  plotted in (b), (c), and (d), where the values of $\lambda$ to which they correspond are
  indicated. The values of the projection dimensions $m_1 = m_2$, the radius $r$, and the
  running time for the proofs are: (b) $m_1 = m_2 = 32$, $r = 2.68062 \times 10^{-12}$,
  running time $40.90$ seconds; (c) $m_1 = m_2 = 61$, $r = 2.13383 \times 10^{-11}$,
  running time $226.94$ seconds; (d) $m_1 = m_2 = 62$, $r = 2.27999 \times 10^{-11}$,
  running time $236.68$ seconds.}
  \label{fig:solns_branch2}
\end{figure}

\begin{figure}[!htbp]
  \centering
  \begin{subfigure}[t]{0.5\textwidth}
      \centering
      \includegraphics[width=0.95\linewidth]{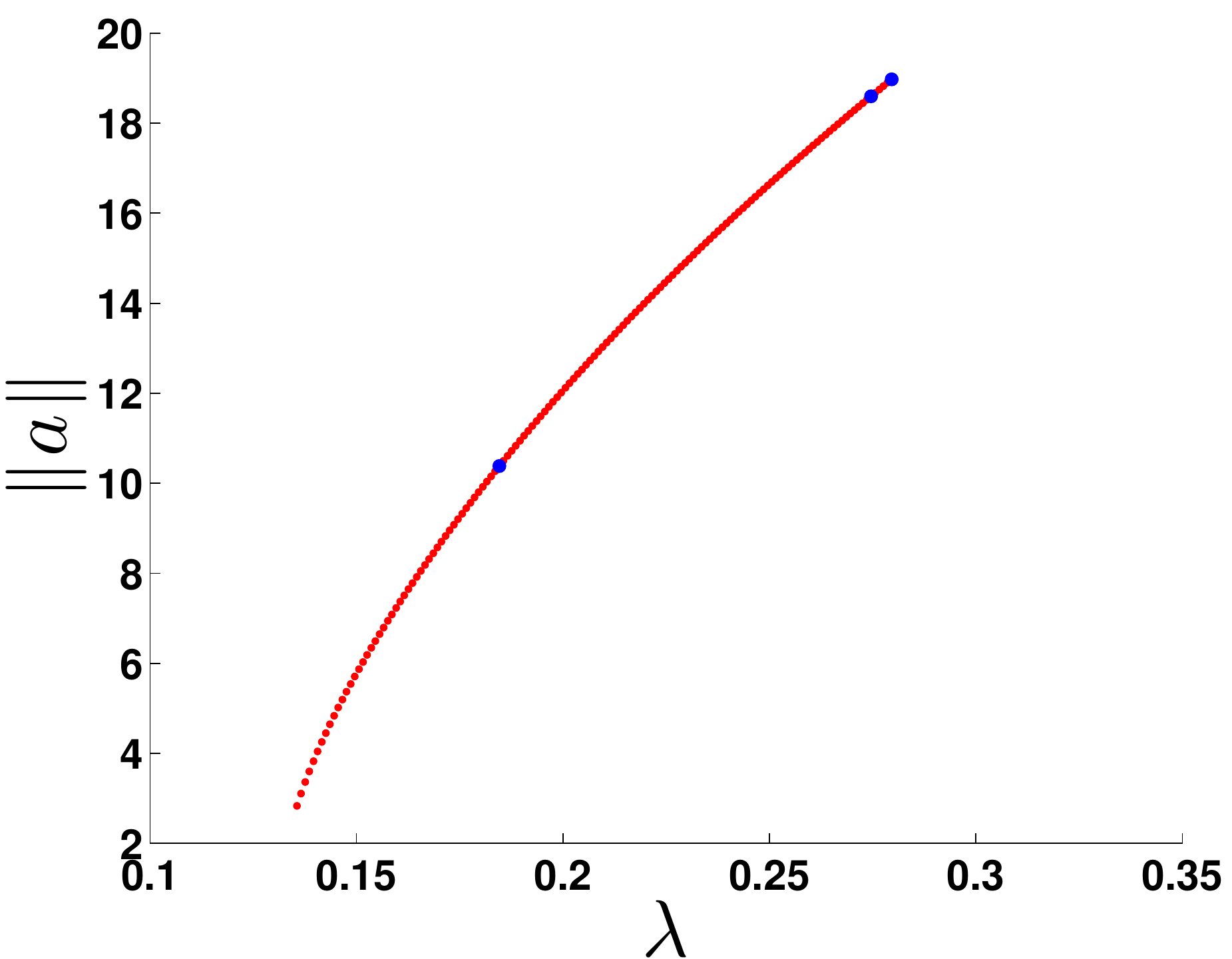}
      \caption{Branch of solutions.}
  \end{subfigure}%
  \begin{subfigure}[t]{0.5\textwidth}
      \centering
      \includegraphics[width=0.95\linewidth]{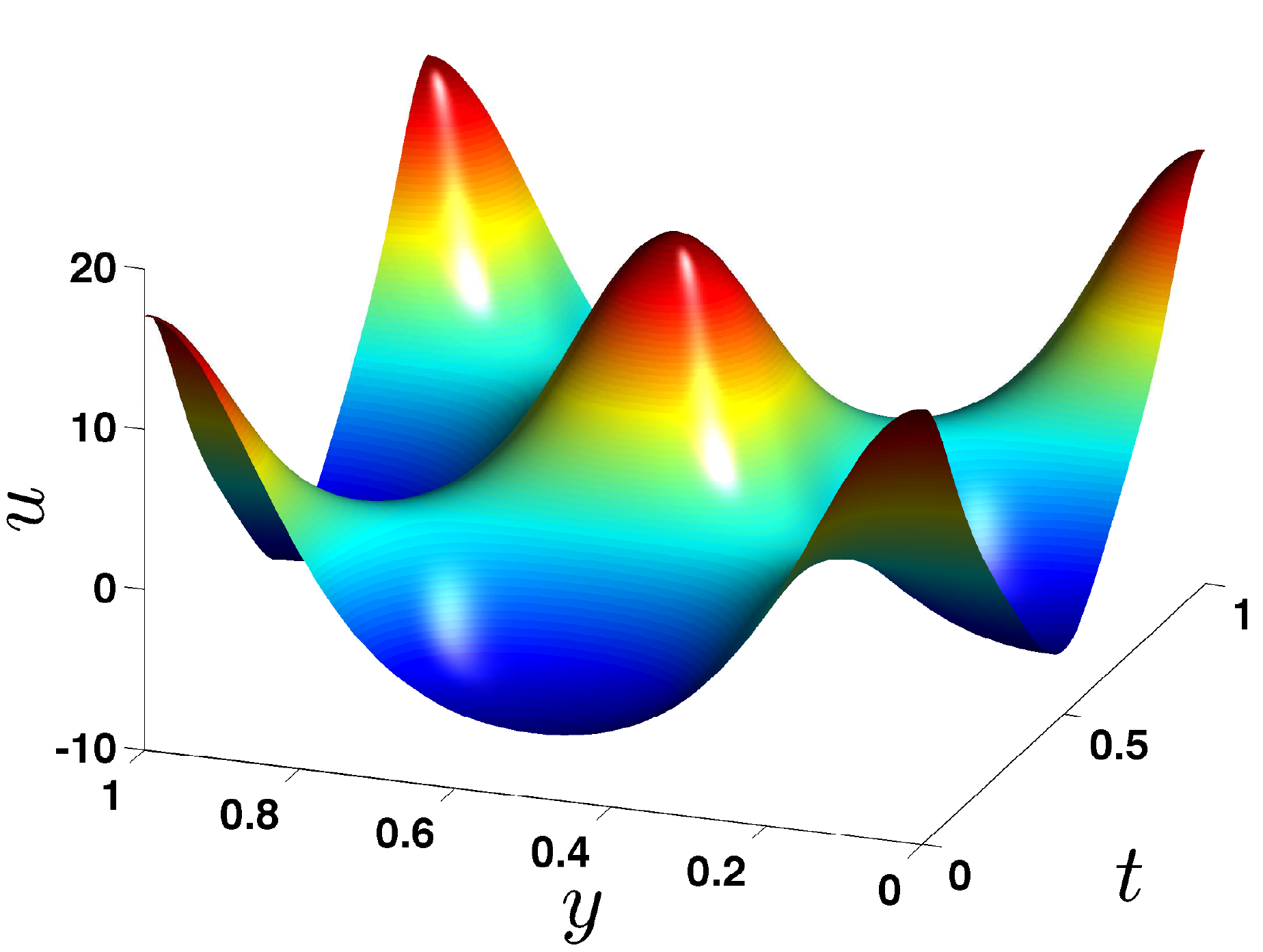}
      \caption{$\lambda = 0.1846$}
  \end{subfigure}
  \begin{subfigure}[t]{0.5\textwidth}
      \centering
      \includegraphics[width=0.95\linewidth]{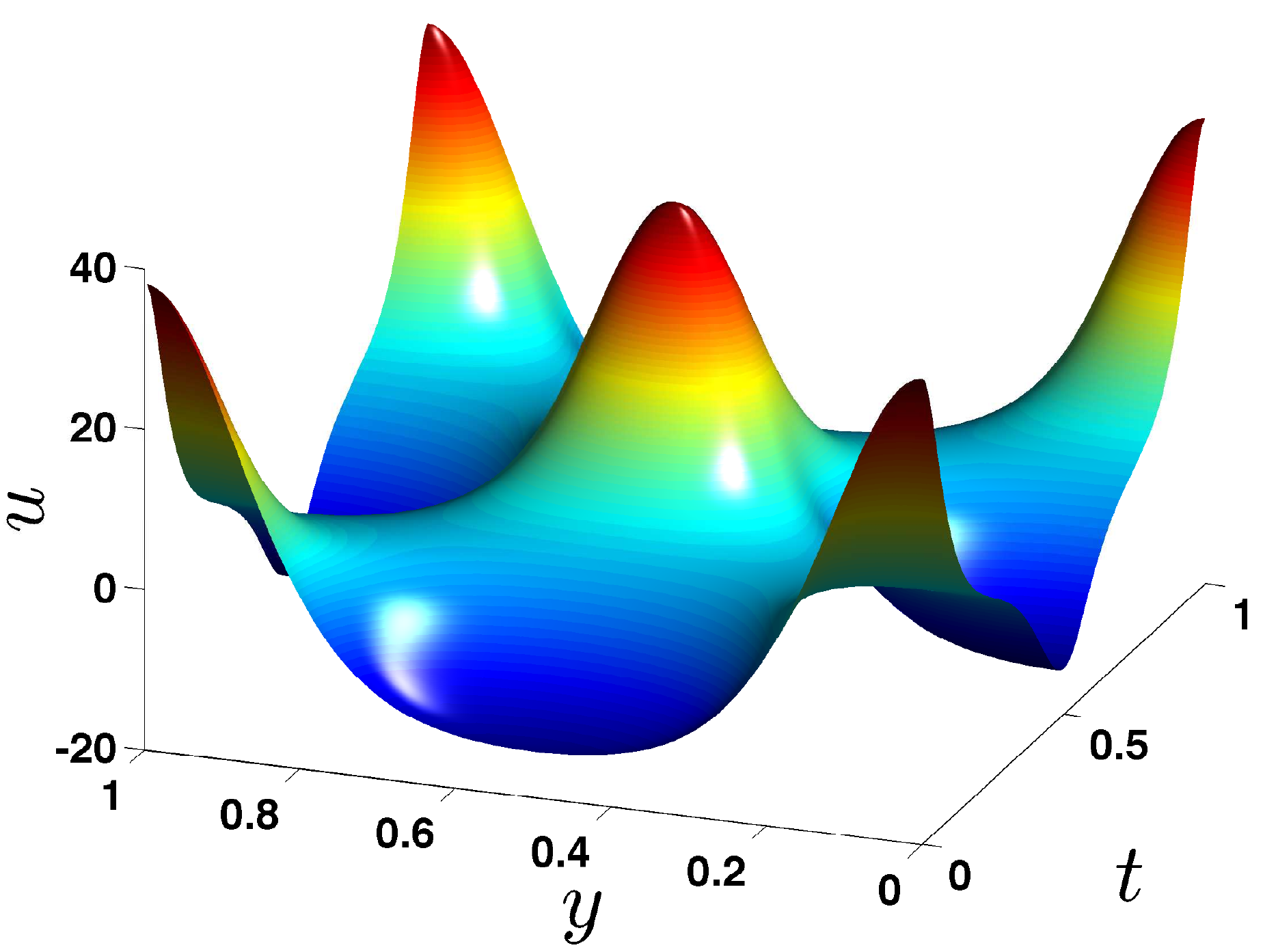}
      \caption{$\lambda = 0.2746$}
  \end{subfigure}%
  \begin{subfigure}[t]{0.5\textwidth}
      \centering
      \includegraphics[width=0.95\linewidth]{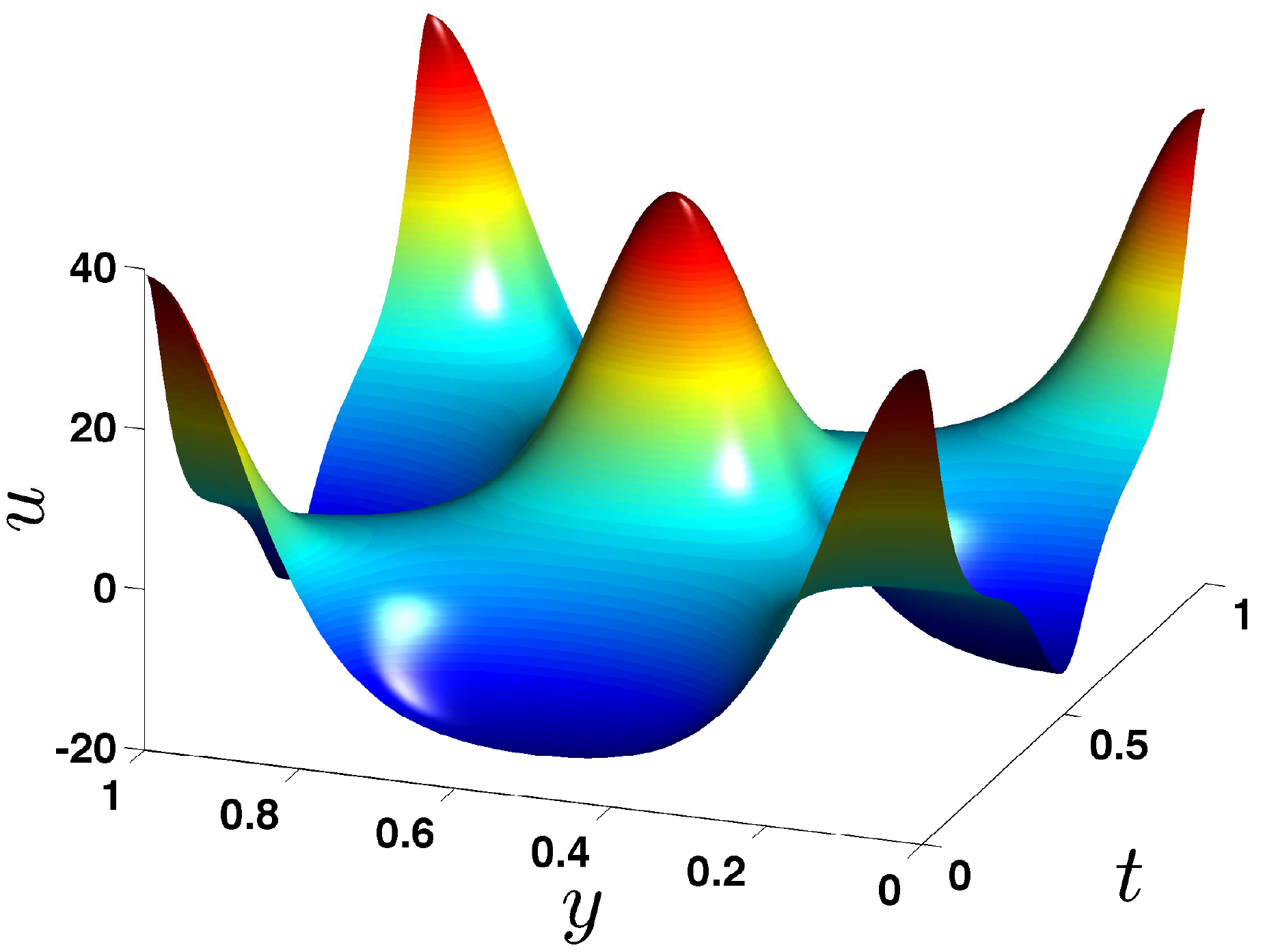}
      \caption{$\lambda = 0.2796$}
  \end{subfigure}%
  \caption{In (a) we plot the branch of numerically computed solutions from which
  we proved, in Theorem~\ref{thm:solns_branch3}, the existence of three periodic
  orbits. These solutions correspond to the larger dots on the curve, and are
  plotted in (b), (c), and (d), where the values of $\lambda$ to which they correspond are
  indicated. The values of the projection dimensions $m_1 = m_2$, the radius $r$, and the
  running time for the proofs are: (b) $m_1 = m_2 = 32$, $r = 3.70605 \times 10^{-12}$,
  running time $40.63$ seconds; (c) $m_1 = m_2 = 61$, $r = 1.56690 \times 10^{-11}$,
  running time $225.13$ seconds; (d) $m_1 = m_2 = 62$, $r = 1.67252 \times 10^{-11}$,
  running time $237.36$ seconds.}
  \label{fig:solns_branch3}
\end{figure}

\begin{figure}[!htbp]
  \centering
  \begin{subfigure}[t]{0.5\textwidth}
      \centering
      \includegraphics[width=0.95\linewidth]{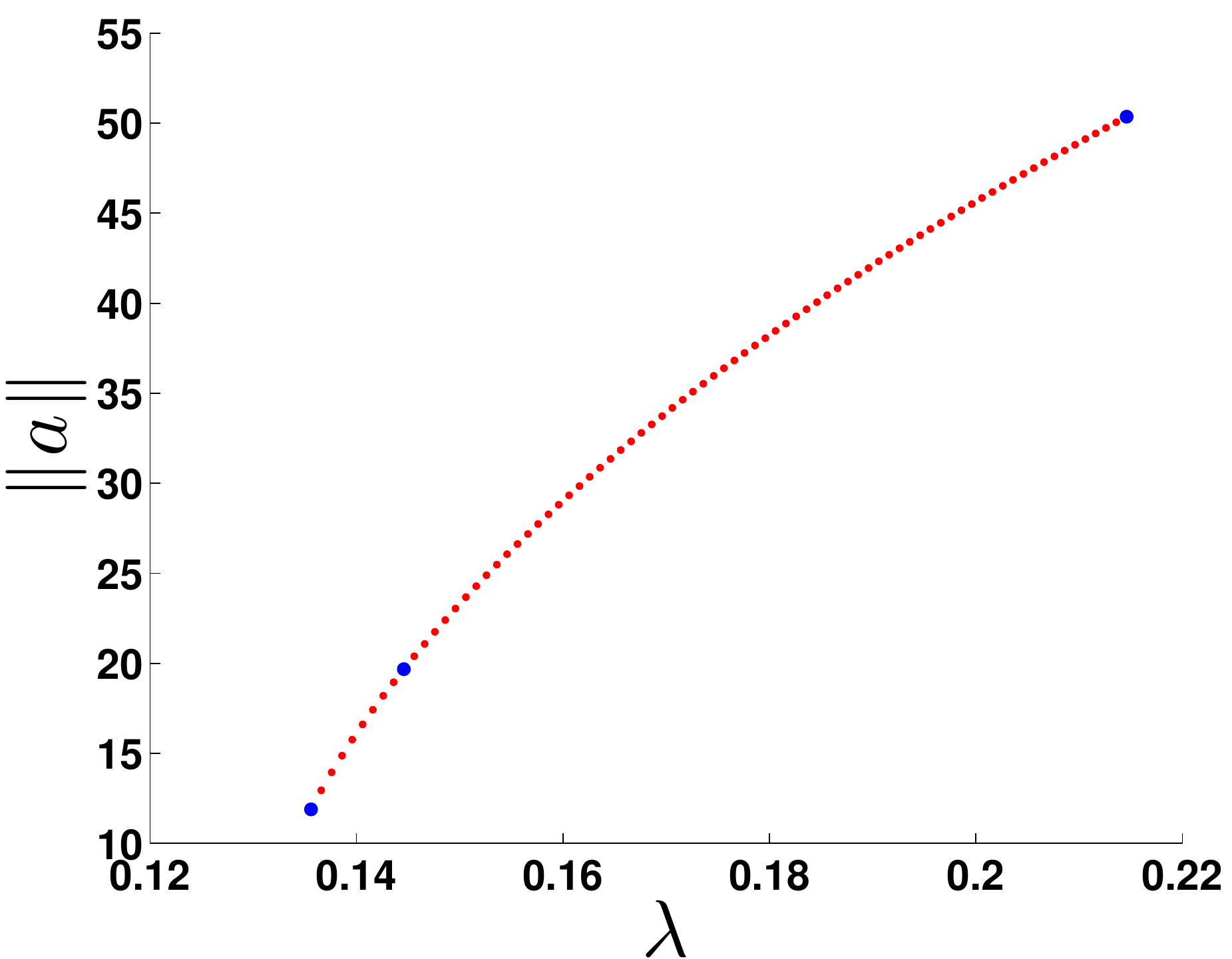}
      \caption{Branch of solutions.}
  \end{subfigure}%
  \begin{subfigure}[t]{0.5\textwidth}
      \centering
      \includegraphics[width=0.95\linewidth]{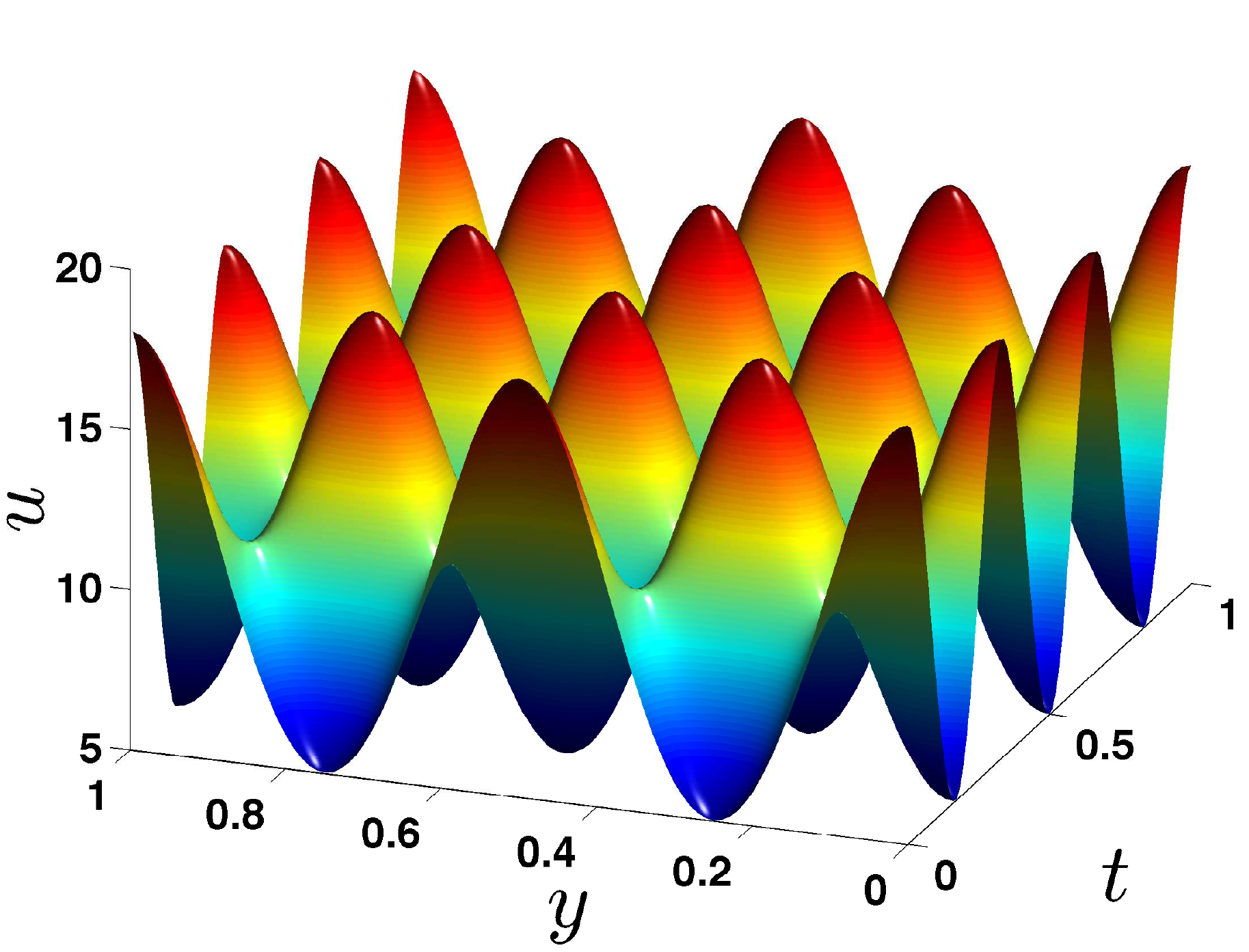}
      \caption{$\lambda = 0.1356$}
  \end{subfigure}
  \begin{subfigure}[t]{0.5\textwidth}
      \centering
      \includegraphics[width=0.95\linewidth]{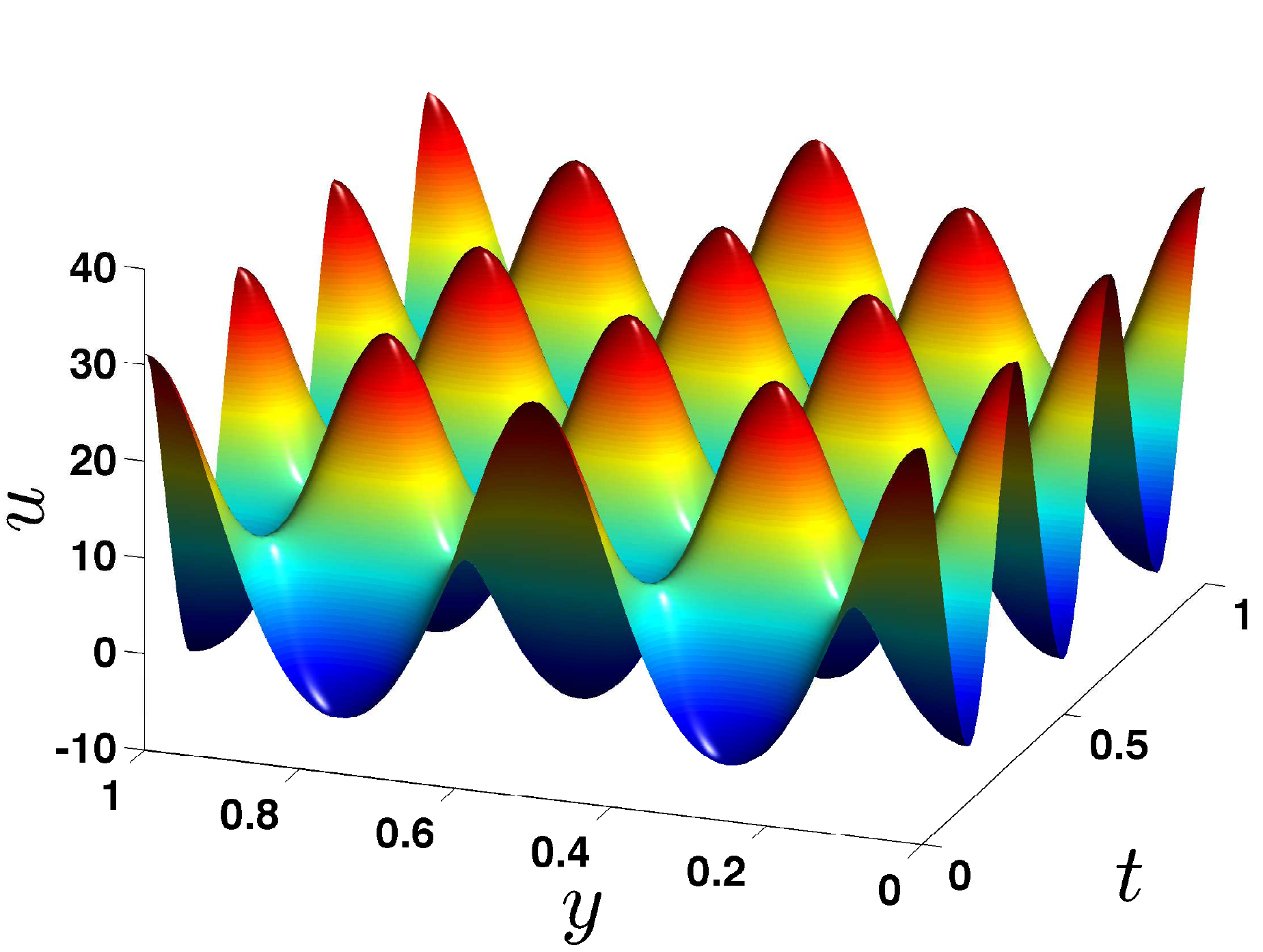}
      \caption{$\lambda = 0.1446$}
  \end{subfigure}%
  \begin{subfigure}[t]{0.5\textwidth}
      \centering
      \includegraphics[width=0.95\linewidth]{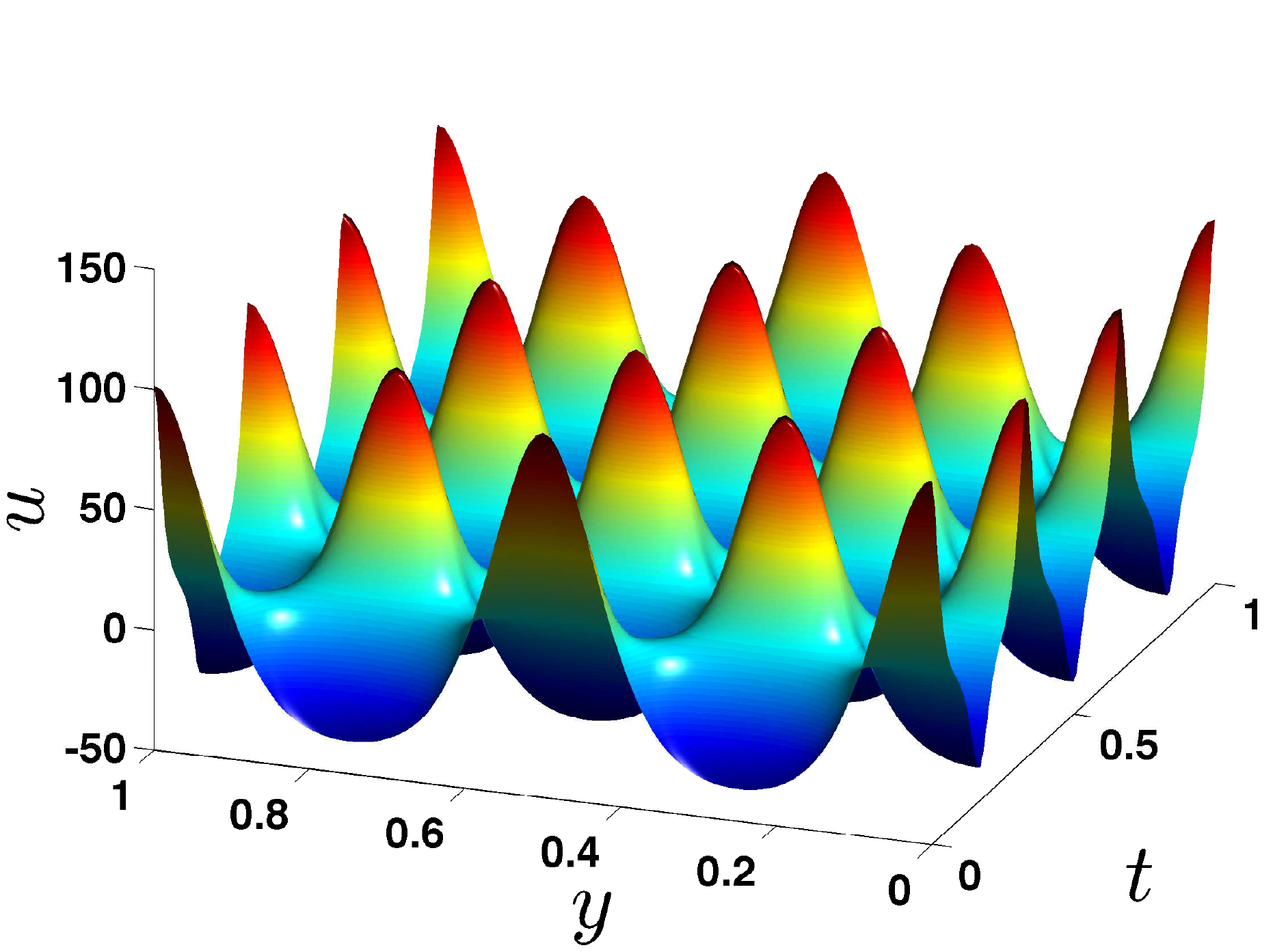}
      \caption{$\lambda = 0.2146$}
  \end{subfigure}%
  \caption{In (a) we plot the branch of numerically computed solutions from which
  we proved, in Theorem~\ref{thm:solns_branch4}, the existence of three periodic
  orbits. These solutions correspond to the larger dots on the curve, and are
  plotted in (b), (c), and (d), where the values of $\lambda$ to which they correspond are
  indicated. The values of the projection dimensions $m_1 = m_2$, the radius $r$, and the
  running time for the proofs are: (b) $m_1 = m_2 = 30$, $r = 2.16093 \times 10^{-10}$,
  running time $35.03$ seconds; (c) $m_1 = m_2 = 61$, $r = 4.37571 \times 10^{-11}$,
  running time $224.99$ seconds; (d) $m_1 = m_2 = 69$, $r = 3.03211 \times 10^{-4}$,
  running time $487.79$ seconds.}
  \label{fig:solns_branch4}
\end{figure}

\subsection{Computing the \boldmath$C^0$\unboldmath-error and \boldmath$L^2$\unboldmath-error bounds}

The error bound in the weighted $\ell^1$ Banach space of space-time Fourier coefficients, as provided by Lemma~\ref{lem:rad_poly}, may not be the most indicative 
quantification of how close the solution $\tx$ is actually from the numerical approximation $\bx$, i.e. the predictors. Here we present how more classical $C^0$ and $L^2$ errors can be obtained. 
Assume that $\tx$ lies in the interior of $B_r(\bx)$, that is $\| \tx - \bx \| <r$. Denote 
\[
\eps(t,y) \bydef \sum_{\bfk \in \mathbb{Z}^2} \tx_\bfk e^{i L k_1t}e^{i k_2 y}
-\sum_{\bfk \in \mathbb{Z}^2} \bx_\bfk e^{i L k_1t}e^{i k_2 y}
= \sum_{\bfk \in \mathbb{Z}^2} (\tx_\bfk - \bx_\bfk  ) e^{i L k_1t}e^{i k_2 y}.
\]
To compute the $C^0$-error, we use Lemma~\ref{lem:norm_comparisons} to get that
\begin{align*}
\sup_{t \in \R} \| \eps(t,\cdot) \|_{C^0} & =  \sup_{t \in \R} \sup_{y \in [0,1]}
\left| \sum_{\bfk \in \mathbb{Z}^2} (\tx_\bfk - \bx_\bfk  ) e^{i L k_1t}e^{i 2 \pi k_2 y} \right|
\le \sum_{\bfk \in \mathbb{Z}^2}  \left| \tx_\bfk - \bx_\bfk   \right| \\
& \le \sum_{\bfk \in \mathbb{Z}^2}  \left| \tx_\bfk - \bx_\bfk   \right| \nu^{|k|}
= \| {\rm sym}(\tx - \bx) \|_\nu^* \le 4 \| \tx - \bx \|_\nu = 4 r.
\end{align*}
For the $L^2$-error, we get
\begin{align*}
\sup_{t \in \R} \| \eps(t,\cdot) \|_{L^2} 
& = \sup_{t \in \R} \sqrt{
\sum_{k_2 \in \mathbb{Z}} \left(
\sum_{k_1 \in \mathbb{Z}} (\bx_\bfk-\tx_\bfk) e^{i L k_1t} 
\right)^2 } \\
& \le \sum_{k_2 \in \mathbb{Z}} 
\sum_{k_1 \in \mathbb{Z}} |\bx_\bfk-\tx_\bfk|
\le 4 \| \tx - \bx \|_\nu = 4 r.
\end{align*}

\subsubsection*{Acknowledgements}

Marcio Gameiro was partially supported by FAPESP grants 2013/07460-7 and 2013/50382-7,
and by CNPq grant 305860/2013-5, Brazil. Jean-Philippe Lessard was partially supported by an NSERC Discovery Grant and by a FAPESP-CALDO grant.

\bibliographystyle{unsrt}

%
%
%
%
%
%
%
%

\end{document}